\documentclass[12pt, reqno]{amsart}
\usepackage{eurosym}
\usepackage{amsmath, amsthm, amscd, amsfonts, amssymb, graphicx, color}
\usepackage[bookmarksnumbered, colorlinks, plainpages]{hyperref}

\hypersetup{colorlinks=true,linkcolor=red, anchorcolor=green, citecolor=cyan, urlcolor=red, filecolor=magenta, pdftoolbar=true}
\newtheorem{theorem}{Theorem}[section]

\newtheorem{proposition}[theorem]{Proposition}

\theoremstyle{definition}
\newtheorem{definition}[theorem]{Definition}

\theoremstyle{remark}
\newtheorem{remark}[theorem]{Remark}
\numberwithin{equation}{section}
\input{tcilatex}

\begin{document}
\title[Positive ideals of multilinear operators ]{Positive ideals of
multilinear operators }
\author{ Athmane Ferradi, Abdelaziz Belaada and Khalil Saadi}
\date{}
\dedicatory{Laboratoire d'Analyse Fonctionnelle et G\'{e}om\'{e}trie des
Espaces, Universit\'{e} de M'sila, Alg\'{e}rie.\\
athmane.ferradi@univ-msila.dz\\
abdelaziz.belaada@univ-msila.dz\\
khalil.saadi@univ-msila.dz}

\begin{abstract}
We introduce and explore the concept of positive ideals for both linear and
multilinear operators between Banach lattices. This paper delineates the
fundamental principles of these new classes and provides techniques for
constructing positive multi-ideals from given positive ideals. Furthermore,
we present an example of a positive multi-ideal by introducing a new class,
referred to as positive $(p_{1},...,p_{m};r)$-dominated multilinear
operators. We establish a natural analogue of the Pietsch domination theorem
and Kwapie\'{n}'s factorization theorem within this class.
\end{abstract}

\maketitle

\setcounter{page}{1}


\let\thefootnote\relax\footnote{
\par
{}\textit{Key words and phrases. }Banach lattice, Positive $p$-summing
operators, Cohen positive strongly $p$-summing multilinear operators,
Positive $(p,q)$-dominated operators, Positive left ideal, Positive right
ideal, Positive ideal, Positive left multi-ideal, Positive right
multi-ideal, Positive multi-ideal, Positive $(p_{1},\ldots ,p_{m};r)$%
-dominated.}

\section{\textsc{Introduction and preliminaries}}

The study of ideals for linear and multilinear operators between Banach
spaces has been extensively explored by various researchers, \cite{BPR07,
defl, DJT95, Gei84, Pml, PIETSCHoi}. In recent years, attention has also
focused on the study of classes of operators defined on Banach lattice
spaces, including positive $p$-summing \cite{Bla87}, positive strongly $p$%
-summing \cite{AB14} and positive $(p,q)$-dominated operators \cite{CBD21}.
The extension of these classes to the multilinear case has been widely
studied, such as Cohen positive strongly $p$-summing multilinear operators 
\cite{BB18}, positive Cohen $p$-nuclear multilinear operators \cite{BBH21}
and factorable positive strongly $p$-summing multilinear operators \cite%
{BBR23}. However, it is important to note that these (positive) classes are
not necessarily considered as operator ideals. The primary objective of this
paper is to introduce and analyze the concepts of positive ideals of linear
operators and positive multi-ideals of multilinear operators. Building on
the definition of multilinear ideals, we aim to explore the positive setting
by examining various classes of positive linear and multilinear operators.
This theoretical framework demonstrates that several previously known
classes of linear and multilinear operators can be classified as positive
operator ideals. We also propose certain construction methods for
generalizing positive multi-ideals. Inspired by the recent work of Chen et
al. \cite{CBD21}, where they introduced and studied the class of positive $%
(p,q)$-dominated linear operators, we extend this concept to the multilinear
case, which we refer to as positive $(p_{1},...,p_{m};r)$-dominated
multilinear operators. The corresponding space is denoted by $\mathcal{D}%
_{\left( p_{1},\ldots ,p_{m};r\right) }^{+}$, which forms a positive
multi-ideal. We provide a proof of a natural analog of Pietsch's domination
theorem within this extended class. Additionally, we proceed with a
presentation of Kwapie\'{n}'s factorization theorem, showing that the space $%
\mathcal{D}_{\left( p_{1},\ldots ,p_{m};r\right) }^{+}$ can be factorized as
follows%
\begin{equation*}
\mathcal{D}_{\left( p_{1},\ldots ,p_{m};r\right) }^{+}=\mathcal{D}_{r^{\ast
}}^{m+}\left( \Pi _{p_{1}}^{+},\ldots ,\Pi _{p_{m}}^{+}\right) ,
\end{equation*}%
where $\Pi _{p_{j}}^{+}$ $\left( 1\leq j\leq m\right) $ is the Banach space
of positive $p$-summing operators, and $\mathcal{D}_{r^{\ast }}^{m+}$ is the
Banach space of Cohen positive strongly $r^{\ast }$-summing multilinear
operators. We conclude this paper with a special case by considering $%
r=\infty ,$ i.e., $1/p=1/p_{1}+...+1/p_{m}$. This class, denoted as positive 
$(p_{1},...,p_{m})$-dominated, constitutes a positive Banach right
multi-ideal.

The paper is structured as follows.

In section 1, we give a brief overview of the basic concepts and
terminologies that are important for our work. We also recall the definition
of positive $p$-summing, Cohen positive strongly $p$-summing and positive $%
(p,q)$-dominated linear operators.

In Section 2, we establish the foundations for positive left ideals, denoted 
$\mathcal{B}_{L}^{+}$ , and right ideals, denoted $\mathcal{B}_{R}^{+}$ , of
linear operators. This conceptual basis can of course be transferred to the
multilinear case. We introduce the composition method to generate a positive
right ideal of multilinear operators from a given positive right ideal of an
operator. The multilinear operator $T$ belongs to the right multi-ideal $%
\mathcal{M}_{R}^{+}$ if and only if its linearization $T_{L}$ belongs to $%
\mathcal{B}_{R}^{+}$. Moreover, we introduce the factorization method to
generate a positive left ideal of multilinear operators from given positive
operator left ideals $\mathcal{B}_{1,L}^{+},...,\mathcal{B}_{m,L}^{+}$. This
section is complemented by illustrative examples of positive left and right
ideals.

In Section 3, we introduce the concept of positive $(p_{1},...,p_{m};r)$%
-dominated multilinear operators. These operators satisfy the Pietsch
factorization theorem and are a good example of a positive multi-ideal.%
\newline

In this paper, we use the notations $E,F,G,E_{1},...,E_{m}$ for Banach
lattices and $X,Y,X_{1},...,X_{m}$ for Banach spaces over $\mathbb{R}$ or $%
\mathbb{C}$. We denote by $\mathcal{L}(X;Y)$ the Banach space of bounded
linear operators from $X$ to $Y$.\ By $B_{X}$ we denote the closed unit ball
of $X$ and by $X^{\ast }$ its topological dual. For $1\leq p\leq \infty $,
let $p^{\ast }$ be its conjugate, i.e., $1/p+1/p^{\ast }=1$. Let $E$ be a
Banach lattice with norm $\left\Vert .\right\Vert $ and order $\leq $. We
denote by $E^{+}$ the positive cone of $E$ , i.e., $E^{+}=\{x\in
E:x\geqslant 0\}.$ For $x\in E$ let $x^{+}:=\sup \{x,0\}\geq 0\ $and $%
x^{-}:=\sup \{-x,0\}\geq 0$ be the positive part and the negative part of $x$%
, respectively. For each $x\in E,$ we have $x=x^{+}-x^{-}$ and $%
|x|=x^{+}+x^{-}.$ The dual $E^{\ast }$ of a Banach lattice $E$ is a Banach
lattice with the natural order 
\begin{equation*}
x_{1}^{\ast }\leq x_{2}^{\ast }\Leftrightarrow \langle x,x_{1}^{\ast
}\rangle \leq \langle x,x_{2}^{\ast }\rangle ,\forall x\in E^{+}.
\end{equation*}%
Recall that a bounded linear operator $T:E\rightarrow F$ is called positive
if $T(x)\in F^{+}$, whenever $x\in E^{+}$. Let $\mathcal{L}^{+}(E;F)$ be the
set of all positive operators from $E$ to $F$. A linear operator $T$ is
called \textit{regular} if there exist $T_{1},T_{2}\in \mathcal{L}^{+}(E;F)$
such that $T=T_{1}-T_{2}.$ We denote by $\mathcal{L}^{r}(E;F)$ the vector
space of regular operators from $E$ to $F.$ It is easy to see that the
vector space $\mathcal{L}^{r}(E;F)$ is generated by positive operators. We
equip $\mathcal{L}^{r}(E;F)$ with the norm, which is defined as%
\begin{equation*}
\left\Vert T\right\Vert _{r}=\inf \left\{ \left\Vert S\right\Vert :S\in 
\mathcal{L}^{+}(E;F),\left\vert T\left( x\right) \right\vert \leq S\left(
x\right) ,x\in E^{+}\right\} ,
\end{equation*}%
then $\mathcal{L}^{r}(E;F)$ becomes a Banach space. By \cite[Section 1.3]%
{MN91}, if $F=\mathbb{R}$, we have $E^{\ast }=\mathcal{L}(E,\mathbb{R})=%
\mathcal{L}^{r}(E,\mathbb{R}).$ By a sublattice of a Banach lattice $E$ we
mean a linear subspace $E_{0}$ of $E$ such that $\sup \left\{ x,y\right\} $
belongs to $E_{0}$ if $x,y\in E_{0}$. The canonical embedding $%
i:E\longrightarrow E^{\ast \ast }$ such that $\left\langle i(x),x^{\ast
}\right\rangle =\left\langle x^{\ast },x\right\rangle $ of $E$ into its
second dual $E^{\ast \ast }$ is an order isometry from $E$ onto a sublattice
of $E^{\ast \ast }$. If we consider $E$ as a sublattice of $E^{\ast \ast }$
we have for $x_{1},x_{2}\in E$ 
\begin{equation*}
x_{1}\leq x_{2}\Longleftrightarrow \left\langle x_{1},x^{\ast }\right\rangle
\leq \left\langle x_{2},x^{\ast }\right\rangle ,\quad \forall x^{\ast }\in
E^{\ast +}.
\end{equation*}%
The spaces $\mathcal{C}(K)$ where $K$ compact and $L_{p}(\mu )$, $(1\leq
p\leq \infty )$ are Banach lattices. For a Banach space $X$, we denote by $%
\ell _{p}^{n}(X)$ the Banach space of all absolutely $p$-summable sequences $%
(x_{i})_{i=1}^{n}\subset X$ with the norm, 
\begin{equation*}
\Vert (x_{i})_{i=1}^{n}\Vert _{p}=\left( \sum_{i=1}^{n}\Vert x_{i}\Vert
^{p}\right) ^{\frac{1}{p}},
\end{equation*}%
and by $\ell _{w,p}^{n}(X)$ the Banach space of all weakly $p$-summable
sequences $(x_{i})_{i=1}^{n}\subset X$ with the norm,%
\begin{equation*}
\Vert (x_{i})_{i=1}^{n}\Vert _{w,p}=\sup_{x^{\ast }\in B_{X^{\ast }}}\left(
\sum_{i=1}^{n}|\langle x^{\ast },x_{i}\rangle |^{p}\right) ^{\frac{1}{p}}.
\end{equation*}%
Consider the case where $X$ is replaced by a Banach lattice $E$, and define 
\begin{equation*}
\ell _{|w|,p}^{n}(E)=\{(x_{i})_{i=1}^{n}\subset E:\left( |x_{i}|\right)
_{i=1}^{n}\in \ell _{w,p}^{n}(E)\}
\end{equation*}%
and $\Vert (x_{i})_{i=1}^{n}\Vert _{|w|,p}=\Vert (|x_{i}|)_{i=1}^{n}\Vert
_{w,p}.$ Let $B_{E^{\ast }}^{+}=\{x^{\ast }\in B_{E^{\ast }}:x^{\ast }\geq
0\}=B_{E^{\ast }}\cap E^{\ast +}$. If $(x_{i})_{i=1}^{n}\subset E^{+}$ , we
have that 
\begin{equation*}
\Vert (x_{i})_{i=1}^{n}\Vert _{|w|,p}=\Vert (x_{i})_{i=1}^{n}\Vert
_{w,p}=\sup_{x^{\ast }\in B_{E^{\ast }}^{+}}\left( \sum_{i=1}^{n}\langle
x^{\ast },x_{i}\rangle ^{p}\right) ^{\frac{1}{p}}.
\end{equation*}

We recall certain classes of positive linear and multilinear operators:

- \textit{Positive }$p$\textit{-summing operator}: Blasco \cite{Bla87}
introduced the positive generalization of $p$-summing operators as follows:
An operator $T:E\longrightarrow X$ is said to be positive $p$-summing ($%
1\leq p<\infty $) if there exists a constant $C>0$ such that the inequality 
\begin{equation}
\left\Vert \left( T\left( x_{i}\right) \right) _{i=1}^{n}\right\Vert
_{p}\leq C\left\Vert \left( x_{i}\right) _{i=1}^{n}\right\Vert _{w,p},
\label{sec1def1}
\end{equation}%
holds 
for all $x_{1},\ldots ,x_{n}\in E^{+}.$ We denote by $\Pi _{p}^{+}(E;X)$,
the space of positive $p$-summing operators from $E$ to $X,$ which is a
Banach space with the norm $\pi _{p}^{+}(T)$ given by the infimum of the
constants $C>0$ that verifying the inequality $(\ref{sec1def1}).$ O.I.
Zhukova \cite{Zhu98}, gives the Pietsch domination theorem. The operator $T$
belongs to $\Pi _{p}^{+}(E;X)$ if and only if there is a Radon probability
measure $\mu $ on the set $B_{E^{\ast }}^{+}$ and a positive constant $C$
such that for every $x\in E^{+}$%
\begin{equation}
\Vert T\left( x\right) \Vert \leq C\left( \int_{B_{E^{\ast }}^{+}}\langle
x,x^{\ast }\rangle ^{p}d\mu (x^{\ast })\right) ^{\frac{1}{p}}.
\label{DomiSumming}
\end{equation}

- \textit{Positive strongly }$p$\textit{-summing operator}:\textit{\ }The
positive generalization of strongly $p$-summing operators, originally
introduced by Cohen \cite{Coh73}, was further developed by Achour and
Belacel \cite{AB14}. An operator $T:X\rightarrow F$ is positive strongly $p$%
-summing ($1<p\leq \infty $) if there exists a constant $C>0$, so that for
all finite sets $(x_{i})_{i=1}^{n}\subset X$ and $(y_{i}^{\ast
})_{i=1}^{n}\subset F^{\ast +}$, we have 
\begin{equation*}
\sum_{i=1}^{n}|\langle T\left( x_{i}\right) ,y_{i}^{\ast }\rangle |\leq
C\Vert (x_{i})_{i=1}^{n}\Vert _{p}\Vert (y_{i}^{\ast })_{i=1}^{n}\Vert
_{w,p^{\ast }}.
\end{equation*}%
The class of all positive strongly $p$-summing operators between $X$ and $F$
is denoted by $\mathcal{D}_{p}^{+}(X;F)$. The infimum of all the constant $C$
in the inequality defines the norm $d_{p}^{+}$ on $\mathcal{D}_{p}^{+}(X;F)$%
. In \cite{AB14} we have: The operator $T$ belongs to $\mathcal{D}%
_{p}^{+}(X;F)$ if and only if there exist a positive constant $C>0$ and
Radon probability measure $\mu $ on $B_{F^{\ast \ast }}^{+}$ such that for
all $x\in X$ and $y^{\ast }\in F^{\ast }$, we have 
\begin{equation}
|\langle T\left( x\right) ,y^{\ast }\rangle |\leq C\Vert x\Vert \left(
\int_{B_{F^{\ast \ast }}^{+}}\langle |y^{\ast }|,\psi \rangle ^{p^{\ast
}}d\mu \right) ^{\frac{1}{p^{\ast }}}.  \label{domPS}
\end{equation}

- \textit{Positive }$(p,q)$\textit{-dominated operator}: The notion of $%
(p,q) $-dominated operator was initiated by Pietsch \cite{PIETSCHoi} and
generalized to positive $(p,q)$-dominated operator by Chen et al. \cite%
{CBD21}. Let $1\leq p,q\leq \infty $ and let $\frac{1}{r}=\frac{1}{p}+\frac{1%
}{q}$. We say that an operator $T$ from a Banach lattice $E$ to a Banach
lattice $F$ is positive $(p,q)$-dominated if there exists a constant $C>0$
such that 
\begin{equation}
\left( \sum_{i=1}^{n}|\langle T\left( x_{i}\right) ,y_{i}^{\ast }\rangle
|^{r}\right) ^{\frac{1}{r}}\leq C\left\Vert \left( x_{i}\right)
_{i=1}^{n}\right\Vert _{w,p}\left\Vert \left( y_{i}^{\ast }\right)
_{i=1}^{n}\right\Vert _{w,q},  \label{D01}
\end{equation}%
for all finite families $(x_{i})_{i=1}^{n}$ in $E^{+}$ and $(y_{i}^{\ast
})_{i=1}^{n}$ in $F^{\ast +}$. The class of all positive $(p,q)$-dominated
operators from $E$ to $F$ is denoted by $\mathcal{D}_{p,q}^{+}(E;F)$. In
this case, we define 
\begin{equation*}
d_{p,q}^{+}(T)=\inf \{C>0\ \text{satisfying\ the\ inequality}(\ref{D01})\}.
\end{equation*}

In \cite[Theorem 3.3]{CBD21}$,$ we have the following result: Let $1\leq
p,q\leq \infty $. An operator $T:E\rightarrow F$ is positive $(p,q)$%
-dominated with positive constant $C$ if and only if there exist a
probability measure $\mu $ on $B_{E^{\ast }}^{+}$ and a probability measure $%
\nu $ on $B_{F^{\ast \ast }}^{+}$ such that 
\begin{equation*}
\left\vert \langle T\left( x\right) ,y^{\ast }\rangle \right\vert \leq
C\left( \int_{B_{E^{\ast }}^{+}}\langle x^{\ast },x\rangle ^{p}d\mu (x^{\ast
})\right) ^{\frac{1}{p}}\left( \int_{B_{F^{\ast \ast }}^{+}}\langle y^{\ast
\ast },y^{\ast }\rangle ^{q}d\nu (y^{\ast \ast })\right) ^{\frac{1}{q}}
\end{equation*}%
for all $x\in E^{+}$ and $y^{\ast }\in F^{\ast +}$.

- \textit{Positive }$p$\textit{-nuclear operator}: If $q=p^{\ast },$ the
definition of positive $(p,p^{\ast })$-dominated operators coincides with
the concept of positive $p$-nuclear operators which introduced and studied
by Bougoutaia et al. \cite{BBH21}. We denote by $\mathcal{N}_{p}^{+}(E;F)$
the space of positive $p$-nuclear operators and $N_{p}^{+}\left( .\right) $
its corresponding norm.

Let $X_{1},...,X_{m},Y$ be Banach spaces and $E_{1},...,E_{m},F$ be Banach
lattices. The space $\mathcal{L}(X_{1},...,X_{m};Y)$ stands for the Banach
space of all bounded multilinear operators from $X_{1}\times ...\times X_{n}$
to $Y.$ We denote the complete projective tensor product of $X_{1},...,X_{m}$
by $X_{1}\widehat{\otimes }_{\pi }...\widehat{\otimes }_{\pi }X_{m}$. For
each $T\in \mathcal{L}(X_{1},...,X_{m};Y)$, we consider its linearization $%
T_{L}:X_{1}\widehat{\otimes }_{\pi }\cdots \widehat{\otimes }_{\pi
}X_{m}\rightarrow Y$, defined as $T_{L}(x^{1}\otimes \cdots \otimes
x^{m})=T(x^{1},\ldots ,x^{m})$. Let $\sigma _{m}:X_{1}\times ...\times
X_{m}\rightarrow X_{1}\widehat{\otimes }_{\pi }\cdots \widehat{\otimes }%
_{\pi }X_{m}$ be the canonical multilinear operator, defined as $\sigma
_{m}(x_{1},...,x_{m})=x_{1}\otimes ...\otimes x_{m}$. Then, we have $%
T=T_{L}\circ \sigma _{m}$. The $m$-linear mapping of finite type $%
T_{f}:X_{1}\times ...\times X_{m}\longrightarrow Y$ is defined as $%
T_{f}\left( x^{1},...,x^{m}\right) =\sum_{i=1}^{k}x_{i,1}^{\ast }\left(
x^{1}\right) ...x_{i,m}^{\ast }\left( x^{m}\right) y_{i}.$ We denote by $%
\mathcal{L}_{f}(X_{1},...,X_{m};Y)$ the space of all $m$-linear mappings of
finite type.

- \textit{Cohen positive strongly }$p$\textit{-summing multilinear operator}%
: This notion was introduced by Bougoutaia and Belacel \cite{BB18} and is a
natural generalization of the multilinear operators studied by Achour and
Mezrag in \cite{AchMez07}: An $m$-linear operator $T:X_{1}\times ...\times
X_{m}\rightarrow F$ is said to be Cohen positive strongly $p$-summing $%
(1<p\leq \infty )$, if there exists a constant $C>0,$ such that for any $%
(x_{i}^{1},...,x_{i}^{m})\in X_{1}\times ...\times X_{m}$ ($1\leq i\leq n$)
and $y_{1}^{\ast },...,y_{n}^{\ast }\in F^{\ast +}$,the following condition
holds 
\begin{equation}
\sum_{i=1}^{n}\left\vert \langle T(x_{i}^{1},...,x_{i}^{m}),y_{i}^{\ast
}\rangle \right\vert \leq C\left( \sum_{i=1}^{n}\prod_{j=1}^{m}\Vert
x_{i}^{j}\Vert ^{p}\right) ^{\frac{1}{p}}\left\Vert \left( y_{i}^{\ast
}\right) _{i=1}^{n}\right\Vert _{w,p^{\ast }}.  \label{nclpos1}
\end{equation}%
We equip the space $\mathcal{D}_{p}^{m+}(X_{1},...,X_{m};F)$ of all Cohen
positive strongly $p$-summing multilinear operators with the norm $%
d_{p}^{m+}(.)$, defined as the smallest constant $C$ for which the
inequality (\ref{nclpos1}) holds.\newline

- \textit{Positive Cohen }$p$\textit{-nuclear multilinear operator}: The
authors in \cite[Definition 2.1]{BBH21} have introduced the class of
positive Cohen $p$-nuclear $m$-linear operators, which are a natural
generalization of the multilinear case studied by Achour and Alouani in \cite%
{AA10}: An $m$-linear operator $T:E_{1}\times ...\times E_{m}\rightarrow F$
is positive Cohen $p$-nuclear $(1<p<\infty )$ if there is a constant $C>0$
such that for any $x_{1}^{j},...,x_{n}^{j}\in E_{j}^{+}$ ($1\leq j\leq m$)
and $y_{1}^{\ast },...,y_{n}^{\ast }\in F^{\ast +}$, we have 
\begin{equation}
\left\vert \sum_{i=1}^{n}\langle T(x_{i}^{1},...,x_{i}^{m}),y_{i}^{\ast
}\rangle \right\vert \leq C\underset{\underset{1\leq j\leq m}{x^{j\ast }\in
B_{E_{j}^{\ast }}^{+}}}{\sup }\left( \sum_{i=1}^{n}\prod_{j=1}^{m}\langle
x_{i}^{j},x^{j\ast }\rangle ^{p}\right) ^{\frac{1}{p}}\left\Vert \left(
y_{i}^{\ast }\right) _{i=1}^{n}\right\Vert _{w,p^{\ast }}.  \label{nclpos}
\end{equation}%
Moreover, the class $\mathcal{N}_{p}^{m+}(E_{1},...,E_{m};F)$ of all
positive Cohen $p$-nuclear $m$-linear operators from $E_{1}\times ...\times
E_{m}$ to $F$ is a Banach space with norm $\eta _{p}^{m+}(.)$, which is the
smallest constant $C$ such that (\ref{nclpos}) holds. If $m=1$, $\mathcal{N}%
_{p}^{+}(E;F)$ is the space of positive Cohen $p$-nuclear operators.

\section{\textsc{Positive operator ideals}}

Several classes of positive operators, including positive $p$-summing, Cohen
positive strong $p$-summing, and positive $(p,q)$-dominated linear
operators, have been introduced and studied. However, these classes are not
considered as ideal operators. Therefore, in this section, we attempt to
introduce the concept of positive operator ideals and propose abstract
methods for generating positive ideals of multilinear operators. In the
context of this definition, we use positive linear operators to determine
the ideal property.

\subsection{\textsc{Positive operator ideal}}

A positive left ideal (or positive left ideal of linear operators), denoted
by $\mathcal{B}_{L}^{+},$ is a subclass of all continuous linear operators
from a Banach space into a Banach lattice such that for every Banach space $%
X $ and Banach lattice $E,$ the components 
\begin{equation*}
\mathcal{B}_{L}^{+}\left( X;E\right) :=\mathcal{L}\left( X;E\right) \cap 
\mathcal{B}_{L}^{+}
\end{equation*}%
satisfy:\newline
$(i)$ $\mathcal{B}_{L}^{+}\left( X;E\right) $ is a linear subspace of $%
\mathcal{L}\left( X;E\right) $ containing the linear mappings of finite rank.%
\newline
$(ii)$ The positive ideal property: If $T\in \mathcal{B}_{L}^{+}\left(
X;E\right) ,u\in \mathcal{L}\left( Y;X\right) $ and $v\in \mathcal{L}%
^{+}\left( E;F\right) $, then $v\circ T\circ u$ is in $\mathcal{B}%
_{L}^{+}\left( Y;F\right) $.\newline
If $\Vert \cdot \Vert _{\mathcal{B}_{L}^{+}}:\mathcal{B}_{L}^{+}\rightarrow 
\mathbb{R}^{+}$ satisfies:\newline
a) $\left( \mathcal{B}_{L}^{+}\left( X;E\right) ,\Vert \cdot \Vert _{%
\mathcal{B}_{L}^{+}}\right) $ is a Banach (quasi-Banach) space for all
Banach space $X$ and Banach lattice $E$.\newline
b) The linear form $T:\mathbb{K}\rightarrow \mathbb{K}$ given by $T\left(
\lambda \right) =\lambda $ satisfies $\left\Vert u\right\Vert _{\mathcal{B}%
_{L}^{+}}=1$,\newline
c) $T\in \mathcal{B}_{L}^{+}\left( X;E\right) ,u\in \mathcal{L}\left(
Y;X\right) $ and $v\in \mathcal{L}^{+}\left( E;F\right) $ then 
\begin{equation*}
\left\Vert v\circ T\circ u\right\Vert _{\mathcal{B}_{L}^{+}}\leq \Vert
v\Vert \Vert T\Vert _{\mathcal{B}_{L}^{+}}\left\Vert u\right\Vert .
\end{equation*}%
The class $\left( \mathcal{B}_{L}^{+},\Vert \cdot \Vert _{\mathcal{B}%
_{L}^{+}}\right) $ is a positive Banach (quasi-Banach) ideal.\newline

\begin{remark}
\label{Remark1}In condition $(ii)$, because every regular operator is a
difference of positive ones, the set $\mathcal{L}^{+}\left( E;F\right) $ can
be replaced by the space $\mathcal{L}^{r}\left( E;F\right) $, and condition $%
(ii)$ remains the same.
\end{remark}

Analogous to the previous approach, we introduce the \textit{positive right
ideal}, denoted $\mathcal{B}_{R}^{+}$, by reversing the roles of the
operators $u$ and $v$. That is, we investigate the composition of positive
linear operators on the right side and linear operators on the left side.
Similarly, we define the \textit{positive ideal}, denoted $\mathcal{B}^{+}$,
by considering only the positive linear operators, with the composition
occurring on both the left and right sides. It is important to note that any
positive right or left ideal is automatically a positive ideal.

\begin{remark}
Every operator ideal is also positive (right, left) ideal.
\end{remark}

\begin{proposition}
\label{Propo1}Let $\mathcal{B}_{L}^{+}$ and $\mathcal{B}_{R}^{+}$ be
positive left and right ideals, respectively. The composition ideal $%
\mathcal{B}_{L}^{+}\circ \mathcal{B}_{R}^{+}$ consists of elements $T$ that
can be factorized as $T=v\circ u,$ where $u$ belongs to $\mathcal{B}%
_{R}^{+}\left( E;X\right) $ and $v$ belongs to $\mathcal{B}_{L}^{+}\left(
X;F\right) $. This construction naturally forms a positive ideal.
\end{proposition}

\begin{proof}
Let $E$ and $F$ be Banach lattices. We will verify that $\mathcal{B}%
_{L}^{+}\circ \mathcal{B}_{R}^{+}\left( E,F\right) $ is a linear subspace.
Let $\lambda \in \mathbb{K}$ and $T\in \mathcal{B}_{L}^{+}\circ \mathcal{B}%
_{R}^{+}(E;F).$ There exist a Banach space $X$ and elements $u_{0}\in 
\mathcal{B}_{R}^{+}\left( E;X\right) ,v_{0}\in \mathcal{B}_{L}^{+}\left(
X;F\right) $ such that $T=v_{0}\circ u_{0}.$ Then $\lambda T=\left( \lambda
v_{0}\right) \circ u_{0}\in \mathcal{B}_{L}^{+}\circ \mathcal{B}%
_{R}^{+}(E;F).$ Now, Let $T_{1},T_{2}\in \mathcal{B}_{L}^{+}\circ \mathcal{B}%
_{R}^{+}(E;F)$ such that there exist Banach spaces $X,Y$ and elements $%
u_{1}\in \mathcal{B}_{R}^{+}\left( E;X\right) ,u_{2}\in \mathcal{B}%
_{R}^{+}\left( E;Y\right) ,v_{1}\in \mathcal{B}_{L}^{+}\left( X;F\right) ,$
and $v_{2}\in \mathcal{B}_{L}^{+}\left( Y;F\right) $ with the following
commutative diagrams: 
\begin{equation*}
\begin{array}{ccc}
E & \overset{T_{1}}{\longrightarrow } & F \\ 
u_{1}\downarrow & \nearrow v_{1} &  \\ 
X &  & 
\end{array}%
\text{ and }%
\begin{array}{ccc}
E & \overset{T_{2}}{\longrightarrow } & F \\ 
u_{2}\downarrow & \nearrow v_{2} &  \\ 
Y &  & 
\end{array}%
\end{equation*}%
We define $A=i_{1}\circ u_{1}+i_{2}\circ u_{2},$ where $i_{1}:X%
\longrightarrow X\times Y$ and $i_{2}:Y\longrightarrow X\times Y$ are given
by $i_{1}\left( x\right) =\left( x,0\right) $ and $i_{2}\left( y\right)
=\left( 0,y\right) .$ We also define $B=v_{1}\circ \pi _{1}+v_{2}\circ \pi
_{2},$ where $\pi _{1}:X\times Y\longrightarrow X$ and $\pi _{2}:X\times
Y\longrightarrow Y$ are given by $\pi _{1}\left( x,y\right) =x$ and $\pi
_{2}\left( x,y\right) =y.$ A simple calculation shows that $%
T_{1}+T_{2}=B\circ A.$ It suffices to show that $A\in \mathcal{B}%
_{R}^{+}\left( E,X\times Y\right) $ and $B\in \mathcal{B}_{L}^{+}\left(
X\times Y,F\right) .$ Indeed, since $u_{j}\in \mathcal{B}_{R}^{+}\left(
E;X\right) $ for $j=1,2,$ we have $i_{j}\circ u_{j}\in \mathcal{B}%
_{R}^{+}\left( E;X\times Y\right) $. Consequently 
\begin{equation*}
A=i_{1}\circ u_{1}+i_{2}\circ u_{2}\in \mathcal{B}_{R}^{+}\left( E;X\times
Y\right) .
\end{equation*}%
Similarly, since $v_{j}\in \mathcal{B}_{L}^{+}\left( X;F\right) $ for $%
j=1,2, $ we have $v_{j}\circ \pi _{j}\in \mathcal{B}_{L}^{+}\left( X\times
Y;F\right) $. Consequently $B=v_{1}\circ \pi _{1}+v_{2}\circ \pi _{2}\in 
\mathcal{B}_{L}^{+}\left( X\times Y;F\right) $. Let $T\in \mathcal{B}(E;F)$
be a finite-rank operator. It can be expressed as a combination of operators
of the form $e^{\ast }b$ where $e^{\ast }\in E^{\ast }$ and $b\in F.$ Let $%
u=e^{\ast }b.$ Define $B:\mathbb{K}\longrightarrow F$ by $B\left( \lambda
\right) =\lambda b=id_{\mathbb{K}}\left( \lambda \right) b.$ Clearly, $B\in 
\mathcal{B}_{L}^{+}\left( \mathbb{K};F\right) $ and define $%
A:E\longrightarrow \mathbb{K}$ by $A\left( x\right) =e^{\ast }\left(
x\right) $ which belongs to $\mathcal{B}_{R}^{+}\left( E;\mathbb{K}\right) .$
Then, we have 
\begin{equation*}
u\left( x\right) =B\circ A\left( x\right) \in \mathcal{B}_{L}^{+}\circ 
\mathcal{B}_{R}^{+}\left( E;F\right) .
\end{equation*}%
By the vector space structure of $\mathcal{B}(E;F)$ it follows that $T\in 
\mathcal{B}_{L}^{+}\circ \mathcal{B}_{R}^{+}\left( E;F\right) .$ Finally, we
verify the ideal property. Let $T=v_{0}\circ u_{0}\in \mathcal{B}%
_{L}^{+}\circ \mathcal{B}_{R}^{+}\left( E;F\right) ,u\in \mathcal{L}%
^{+}\left( D;E\right) $ and $v\in \mathcal{L}^{+}\left( F;G\right) $. Then 
\begin{equation*}
v\circ T\circ u=\left( v\circ v_{0}\right) \circ \left( u_{0}\circ u\right) .
\end{equation*}%
Since $v\circ v_{0}\in \mathcal{B}_{L}^{+}\left( X;G\right) $ and $%
u_{0}\circ u\in \mathcal{B}_{R}^{+}\left( D;X\right) ,$ we obtain $v\circ
T\circ u\in \mathcal{B}_{L}^{+}\circ \mathcal{B}_{R}^{+}\left( E;F\right) .$
\end{proof}

Let $\mathcal{B}_{L}^{+}$ and $\mathcal{B}_{R}^{+}$ be positive Banach left
and right ideals, respectively. If $E$ and $F$ are Banach lattices and $T\in 
\mathcal{B}_{L}^{+}\circ \mathcal{B}_{R}^{+}(E;F),$ then we define%
\begin{equation}
\left\Vert T\right\Vert _{\mathcal{B}_{L}^{+}\circ \mathcal{B}_{R}^{+}}=\inf
\left\{ \left\Vert v\right\Vert _{\mathcal{B}_{L}^{+}}\left\Vert
u\right\Vert _{\mathcal{B}_{R}^{+}}:T=v\circ u\right\} .  \tag{2.1}
\end{equation}%
Therefore, if $T\in \mathcal{B}_{L}^{+}\circ \mathcal{B}_{R}^{+}\left(
E;F\right) ,$ then 
\begin{equation}
\left\Vert T\right\Vert \leq \left\Vert T\right\Vert _{\mathcal{B}%
_{L}^{+}\circ \mathcal{B}_{R}^{+}}.  \tag{2.2}
\end{equation}%
Indeed, let $\varphi \in F^{\ast }$ and $x\in E.$ Consider $B:\mathbb{K}%
\longrightarrow E$ defined by $B\left( \lambda \right) =\lambda x.$ We have $%
\left\Vert B\right\Vert =\left\Vert x\right\Vert $ and%
\begin{equation*}
\varphi \circ T\circ B\left( \lambda \right) =\lambda \left\langle \varphi
,T\left( x\right) \right\rangle .
\end{equation*}%
Thus $\varphi \circ T\circ B=\left\langle \varphi ,T\left( x\right)
\right\rangle id_{\mathbb{K}}.$ We have 
\begin{eqnarray*}
\left\vert \left\langle \varphi ,T\left( x\right) \right\rangle \right\vert
&=&\left\vert \left\langle \varphi ,T\left( x\right) \right\rangle
\right\vert \left\Vert id_{\mathbb{K}}\right\Vert _{\mathcal{B}_{L}^{+}\circ 
\mathcal{B}_{R}^{+}}=\left\Vert \left\langle \varphi ,T\left( x\right)
\right\rangle id_{\mathbb{K}}\right\Vert _{\mathcal{B}_{L}^{+}\circ \mathcal{%
B}_{R}^{+}} \\
&=&\left\Vert \varphi \circ T\circ B\right\Vert _{\mathcal{B}_{L}^{+}\circ 
\mathcal{B}_{R}^{+}}\leq \left\Vert \varphi \right\Vert \left\Vert
T\right\Vert _{\mathcal{B}_{L}^{+}\circ \mathcal{B}_{R}^{+}}\left\Vert
B\right\Vert =\left\Vert \varphi \right\Vert \left\Vert T\right\Vert _{%
\mathcal{B}_{L}^{+}\circ \mathcal{B}_{R}^{+}}\left\Vert x\right\Vert .
\end{eqnarray*}%
Then%
\begin{eqnarray*}
\left\Vert T\left( x\right) \right\Vert &=&\sup_{\varphi \in B_{F^{\ast
}}}\left\vert \left\langle \varphi ,T\left( x\right) \right\rangle
\right\vert \leq \sup_{\varphi \in B_{F^{\ast }}}\left\Vert \varphi
\right\Vert \left\Vert T\right\Vert _{\mathcal{B}_{L}^{+}\circ \mathcal{B}%
_{R}^{+}}\left\Vert x\right\Vert \\
&\leq &\left\Vert T\right\Vert _{\mathcal{B}_{L}^{+}\circ \mathcal{B}%
_{R}^{+}}\left\Vert x\right\Vert
\end{eqnarray*}%
consequently, $\left\Vert T\right\Vert \leq \left\Vert T\right\Vert _{%
\mathcal{B}_{L}^{+}\circ \mathcal{B}_{R}^{+}}.$

\begin{theorem}
\label{Theorem1}If $\mathcal{B}_{L}^{+}$ and $\mathcal{B}_{R}^{+}$ are
positive Banach left and right ideals, respectively. Then%
\begin{equation*}
\left( \mathcal{B}_{L}^{+}\circ \mathcal{B}_{R}^{+},\left\Vert .\right\Vert
_{\mathcal{B}_{L}^{+}\circ \mathcal{B}_{R}^{+}}\right)
\end{equation*}%
is a positive quasi-Banach ideal.
\end{theorem}

\begin{proof}
It is straightforward to show that 
\begin{equation*}
\left\Vert T\right\Vert _{\mathcal{B}_{L}^{+}\circ \mathcal{B}_{R}^{+}}=\inf
\left\{ \left\Vert u\right\Vert _{\mathcal{B}_{R}^{+}}:T=v\circ u\text{ and }%
\left\Vert v\right\Vert _{\mathcal{B}_{L}^{+}}=1\right\} .
\end{equation*}%
Indeed, consider a representation of $T$ as $v_{0}\circ u_{0}.$ Then we can
write $T=(\frac{v_{0}}{\left\Vert v_{0}\right\Vert _{\mathcal{B}_{L}^{+}}}%
)\circ (\left\Vert v_{0}\right\Vert _{\mathcal{B}_{L}^{+}}u_{0}),$ and we
have%
\begin{equation*}
\left\Vert \left\Vert v_{0}\right\Vert _{\mathcal{B}_{L}^{+}}u_{0}\right%
\Vert _{\mathcal{B}_{R}^{+}}\geq \inf \left\{ \left\Vert u\right\Vert _{%
\mathcal{B}_{R}^{+}}:T=v\circ u\text{ and }\left\Vert v\right\Vert _{%
\mathcal{B}_{L}^{+}}=1\right\} .
\end{equation*}%
This implies%
\begin{equation*}
\left\Vert v_{0}\right\Vert _{\mathcal{B}_{L}^{+}}\left\Vert
u_{0}\right\Vert _{\mathcal{B}_{R}^{+}}\geq \inf \left\{ \left\Vert
u\right\Vert _{\mathcal{B}_{R}^{+}}:T=v\circ u\text{ and }\left\Vert
v\right\Vert _{\mathcal{B}_{L}^{+}}=1\right\}
\end{equation*}%
Taking the infimum over all representations $T=$ $v\circ u,$ we find%
\begin{equation*}
\left\Vert T\right\Vert _{\mathcal{B}_{L}^{+}\circ \mathcal{B}_{R}^{+}}\geq
\inf \left\{ \left\Vert u\right\Vert _{\mathcal{B}_{R}^{+}}:T=v\circ u\text{
and }\left\Vert v\right\Vert _{\mathcal{B}_{L}^{+}}=1\right\} .
\end{equation*}%
Now, let $v_{0}\circ u_{0}$ be a representation of $T$ such that $\left\Vert
v_{0}\right\Vert _{\mathcal{B}_{L}^{+}}=1.$ Then $\left\Vert T\right\Vert _{%
\mathcal{B}_{L}^{+}\circ \mathcal{B}_{R}^{+}}\leq \left\Vert
u_{0}\right\Vert _{\mathcal{B}_{R}^{+}}.$ Taking the infimum over all such
representations, we obtain%
\begin{equation*}
\left\Vert T\right\Vert _{\mathcal{B}_{L}^{+}\circ \mathcal{B}_{R}^{+}}\leq
\inf \left\{ \left\Vert u\right\Vert _{\mathcal{B}_{R}^{+}}:T=v\circ u\text{
and }\left\Vert v\right\Vert _{\mathcal{B}_{L}^{+}}=1\right\} .
\end{equation*}%
Next, let $E$ and $F$ be Banach lattices. We will verify that $\left\Vert
.\right\Vert _{\mathcal{B}_{L}^{+}\circ \mathcal{B}_{R}^{+}}$ is a quasi
norm; the rest is trivial. Let $\lambda \in \mathbb{K}$ and $T\in \mathcal{B}%
_{L}^{+}\circ \mathcal{B}_{R}^{+}(E;F).$ There exist a Banach space $X$ and
elements $u_{0}\in \mathcal{B}_{R}^{+}\left( E;X\right) ,v_{0}\in \mathcal{B}%
_{L}^{+}\left( X;F\right) $ such that $T=v_{0}\circ u_{0}.$ Then%
\begin{equation*}
\left\Vert \lambda T\right\Vert _{\mathcal{B}_{L}^{+}\circ \mathcal{B}%
_{R}^{+}}\leq \left\Vert \lambda v_{0}\right\Vert _{\mathcal{B}%
_{L}^{+}}\left\Vert u_{0}\right\Vert _{\mathcal{B}_{R}^{+}}=\left\vert
\lambda \right\vert \left\Vert v_{0}\right\Vert _{\mathcal{B}%
_{L}^{+}}\left\Vert u_{0}\right\Vert _{\mathcal{B}_{R}^{+}},
\end{equation*}%
if we take the infimum over all representations of $T,$ we find $\left\Vert
\lambda T\right\Vert _{\mathcal{B}_{L}^{+}\circ \mathcal{B}_{R}^{+}}\leq
\left\vert \lambda \right\vert \left\Vert T\right\Vert _{\mathcal{B}%
_{L}^{+}\circ \mathcal{B}_{R}^{+}}.$ We check the inverse inequality only
for $\lambda \neq 0.$ Let $v_{0}\circ u_{0}$ be a representation of $\lambda
T.$ Then $T=\frac{v_{0}}{\lambda }\circ u_{0}$ and we have%
\begin{equation*}
\left\Vert T\right\Vert _{\mathcal{B}_{L}^{+}\circ \mathcal{B}_{R}^{+}}\leq
\left\Vert \frac{v_{0}}{\lambda }\right\Vert _{\mathcal{B}%
_{L}^{+}}\left\Vert u_{0}\right\Vert _{\mathcal{B}_{R}^{+}}\leq \frac{1}{%
\left\vert \lambda \right\vert }\left\Vert v_{0}\right\Vert _{\mathcal{B}%
_{L}^{+}}\left\Vert u_{0}\right\Vert _{\mathcal{B}_{R}^{+}}
\end{equation*}%
and taking the infimum over all representations of $\lambda T,$ we find $%
\left\vert \lambda \right\vert \left\Vert T\right\Vert _{\mathcal{B}%
_{L}^{+}\circ \mathcal{B}_{R}^{+}}\leq \left\Vert \lambda T\right\Vert _{%
\mathcal{B}_{L}^{+}\circ \mathcal{B}_{R}^{+}}.$ Now, by $\left( 2.2\right) $
if $\left\Vert T\right\Vert _{\mathcal{B}_{L}^{+}\circ \mathcal{B}%
_{R}^{+}}=0,$ then $T=0.$ Let $T_{1},T_{2}\in \mathcal{B}_{L}^{+}\circ 
\mathcal{B}_{R}^{+}(E;F).$ Following a similar approach to the proof of
Proposition \ref{Propo1}$,$ $T_{1}+T_{2}=B\circ A.$ We can then establish
the following inequalities%
\begin{eqnarray*}
\left\Vert A\right\Vert _{\mathcal{B}_{R}^{+}} &\leq &\left\Vert i_{1}\circ
u_{1}\right\Vert _{\mathcal{B}_{R}^{+}}+\left\Vert i_{2}\circ
u_{2}\right\Vert _{\mathcal{B}_{R}^{+}} \\
&\leq &\left\Vert i_{1}\right\Vert \left\Vert u_{1}\right\Vert _{\mathcal{B}%
_{R}^{+}}+\left\Vert i_{2}\right\Vert \left\Vert u_{2}\right\Vert _{\mathcal{%
B}_{R}^{+}}=\left\Vert u_{1}\right\Vert _{\mathcal{B}_{R}^{+}}+\left\Vert
u_{2}\right\Vert _{\mathcal{B}_{R}^{+}}.
\end{eqnarray*}%
Similarly,%
\begin{eqnarray*}
\left\Vert B\right\Vert _{\mathcal{B}_{L}^{+}} &\leq &\left\Vert v_{1}\circ
\pi _{1}\right\Vert _{\mathcal{B}_{L}^{+}}+\left\Vert v_{2}\circ \pi
_{2}\right\Vert _{\mathcal{B}_{L}^{+}} \\
&\leq &\left\Vert \pi _{1}\right\Vert \left\Vert v_{1}\right\Vert _{\mathcal{%
B}_{L}^{+}}+\left\Vert \pi _{2}\right\Vert \left\Vert v_{2}\right\Vert _{%
\mathcal{B}_{L}^{+}}=\left\Vert v_{1}\right\Vert _{\mathcal{B}%
_{L}^{+}}+\left\Vert v_{2}\right\Vert _{\mathcal{B}_{L}^{+}}.
\end{eqnarray*}%
Now, for each $\varepsilon >0$ we can choose $u_{1},u_{2},v_{1},v_{2}$ such
that 
\begin{equation*}
\left\Vert u_{j}\right\Vert _{\mathcal{B}_{R}^{+}}\leq \left\Vert
T_{j}\right\Vert _{\mathcal{B}_{L}^{+}\circ \mathcal{B}_{R}^{+}}+\varepsilon
,\text{ and }\left\Vert v_{j}\right\Vert _{\mathcal{B}_{L}^{+}}=1\text{ for }%
j=1,2
\end{equation*}%
A simple calculation shows that%
\begin{eqnarray*}
\left\Vert T_{1}+T_{2}\right\Vert _{\mathcal{B}_{L}^{+}\circ \mathcal{B}%
_{R}^{+}} &\leq &\left\Vert A\right\Vert _{\mathcal{B}_{R}^{+}}\left\Vert
B\right\Vert _{\mathcal{B}_{L}^{+}} \\
&\leq &\left( \left\Vert u_{1}\right\Vert _{\mathcal{B}_{R}^{+}}+\left\Vert
u_{2}\right\Vert _{\mathcal{B}_{R}^{+}}\right) \left( \left\Vert
v_{1}\right\Vert _{\mathcal{B}_{L}^{+}}+\left\Vert v_{2}\right\Vert _{%
\mathcal{B}_{L}^{+}}\right) \\
&\leq &2\left( \left\Vert T_{1}\right\Vert _{\mathcal{B}_{L}^{+}\circ 
\mathcal{B}_{R}^{+}}+\left\Vert T_{2}\right\Vert _{\mathcal{B}_{L}^{+}\circ 
\mathcal{B}_{R}^{+}}+2\varepsilon \right)
\end{eqnarray*}%
Since $\varepsilon $ is arbitrary, it follows that%
\begin{equation*}
\left\Vert T_{1}+T_{2}\right\Vert _{\mathcal{B}_{L}^{+}\circ \mathcal{B}%
_{R}^{+}}\leq 2\left( \left\Vert T_{1}\right\Vert _{\mathcal{B}_{L}^{+}\circ 
\mathcal{B}_{R}^{+}}+\left\Vert T_{2}\right\Vert _{\mathcal{B}_{L}^{+}\circ 
\mathcal{B}_{R}^{+}}\right) .
\end{equation*}
\end{proof}

We now provide examples of positive ideals. For each $p\geq 1$, the space of
positive $p$-summing operators $\Pi _{p}^{+}$ forms a positive right ideal,
while the space of Cohen positive strongly $p$-summing operators $\mathcal{D}%
_{p}^{+}$ forms a positive left ideal. Consequently, the composition ideal $%
\mathcal{D}_{p}^{+}\circ \Pi _{p}^{+},$ which is equal to $\mathcal{N}%
_{p}^{+},$ the space of positive $p$-nuclear operators, forms a positive
ideal. As a result, $\left( \mathcal{N}_{p}^{+},N_{p}^{+}\left( .\right)
\right) $ is a positive Banach ideal, while $\left( \mathcal{N}%
_{p}^{+},\left\Vert .\right\Vert _{\mathcal{D}_{p}^{+}\circ \Pi
_{p}^{+}}\right) $ is a positive quasi-Banach ideal.

\subsection{\textsc{Positive left multi-ideal}}

We extend the previous concepts to the multilinear case by introducing a new
definition for positive ideals of multilinear operators. This new concept is
an extension of multi-ideals and utilizes the techniques introduced in \cite%
{BPR07,Pml}. To begin with, we introduce the notion of positive left
multi-ideal. Moreover, we propose a composition method that allows us to
construct a positive left ideal of multilinear operators from a given
positive left ideal.

\begin{definition}
A positive left multi-ideal (or positive left ideal of multilinear
operators), denoted by $\mathcal{M}_{L}^{+},$ is a subclass of all
continuous multilinear operators from Banach spaces into a Banach lattice
such that for all $m\in \mathbb{N}^{\ast }$, Banach spaces $X_{1},\ldots
,X_{m}$ and Banach lattice $F,$ the components 
\begin{equation*}
\mathcal{M}_{L}^{+}(X_{1},...,X_{m};F):=\mathcal{L}(X_{1},...,X_{m};F)\cap 
\mathcal{M}_{L}^{+}
\end{equation*}%
satisfy:\newline
$(i)$ $\mathcal{M}_{L}^{+}(X_{1},...,X_{m};F)$ is a linear subspace of $%
\mathcal{L}(X_{1},...,X_{m};F)$ which contains the $m$-linear mappings of
finite rank.\newline
$(ii)$ The positive ideal property: If $T\in \mathcal{M}_{L}^{+}\left(
X_{1},\ldots ,X_{m};F\right) ,u_{j}\in \mathcal{L}\left( Y_{j};X_{j}\right) $
for $j=1,\ldots ,m$ and $v\in \mathcal{L}^{+}(F;E)$, then $v\circ T\circ
\left( u_{1},\ldots ,u_{m}\right) $ is in $\mathcal{M}_{L}^{+}\left(
Y_{1},\ldots ,Y_{m};E\right) $.\newline
If $\Vert \cdot \Vert _{\mathcal{M}_{L}^{+}}:\mathcal{M}_{L}^{+}\rightarrow 
\mathbb{R}^{+}$ satisfies:\newline
a) $\left( \mathcal{M}_{L}^{+}(X_{1},...,X_{m};F),\Vert \cdot \Vert _{%
\mathcal{M}_{L}^{+}}\right) $ is a Banach (quasi-Banach) space for all
Banach spaces $X_{1},\ldots ,X_{m}$ and Banach lattice $F$.\newline
b) The $m$ -linear form $T^{m}:\mathbb{K}^{m}\rightarrow \mathbb{K}$ given
by $T^{m}\left( \lambda ^{1},\ldots ,\lambda ^{m}\right) =\lambda ^{1}\ldots
\lambda ^{m}$ satisfies $\left\Vert T^{m}\right\Vert _{\mathcal{M}%
_{L}^{+}}=1 $ for all $m$,\newline
c) $T\in \mathcal{M}_{L}^{+}\left( X_{1},\ldots ,X_{m};F\right) ,u_{j}\in 
\mathcal{L}\left( Y_{j};X_{j}\right) $ for $j=$ $1,\ldots ,m$ and $v\in 
\mathcal{L}^{+}(F;E)$ then 
\begin{equation*}
\left\Vert v\circ T\circ \left( u_{1},\ldots ,u_{m}\right) \right\Vert _{%
\mathcal{M}_{L}^{+}}\leq \Vert v\Vert \Vert T\Vert _{\mathcal{M}%
_{L}^{+}}\left\Vert u_{1}\right\Vert \ldots \left\Vert u_{m}\right\Vert
\end{equation*}%
we say that $\left( \mathcal{M}_{L}^{+},\Vert \cdot \Vert _{\mathcal{M}%
_{L}^{+}}\right) $ is a positive Banach (quasi-Banach) right multi-ideal.
When $m=1$, it specifically corresponds to the case of a positive right
ideal.
\end{definition}

\begin{remark}
\label{Remark2}As mentioned in the remark \ref{Remark1}$,$ we can substitute
the set $\mathcal{L}^{+}\left( F;E\right) $ with the space $\mathcal{L}%
^{r}\left( F;E\right) $. Indeed, consider $v\in \mathcal{L}^{r}(F;E)$ such
that $v=v_{1}-v_{2},$ where $v_{1},v_{2}\subset \mathcal{L}^{+}\left(
E;F\right) $. Now, we can write%
\begin{eqnarray*}
v\circ T\circ \left( u_{1},\ldots ,u_{m}\right) &=&\left( v_{1}-v_{2}\right)
\circ T\circ \left( u_{1},\ldots ,u_{m}\right) \\
&=&v_{1}\circ T\circ \left( u_{1},\ldots ,u_{m}\right) -v_{2}\circ T\circ
\left( u_{1},\ldots ,u_{m}\right) .
\end{eqnarray*}%
By the linearity of $\mathcal{M}_{L}^{+}\left( Y_{1},\ldots ,Y_{m};E\right) $%
, we conclude that $v\circ T\circ \left( u_{1},\ldots ,u_{m}\right) \in 
\mathcal{M}_{L}^{+}\left( Y_{1},\ldots ,Y_{m};E\right) $.\newline
\end{remark}

\begin{remark}
It is evident that all multi-ideals are also positive left multi-ideals.
\end{remark}

Here is an example of a positive left multi-ideal. In \cite{BB18}, the
authors introduced the class $\mathcal{D}_{p}^{m+}$ of Cohen positive
strongly $p$-summing multilinear operators.

\begin{proposition}
The class $(\mathcal{D}_{p}^{m+},d_{p}^{m+}(.))$ is a\ positive Banach left
multi-ideal.
\end{proposition}

\begin{proof}
$(i)$ $\mathcal{D}_{p}^{m+}(X_{1},...,X_{m};F)$ is a linear subspace of $%
\mathcal{L}(X_{1},...,X_{m};F)$ that contains the $m$-linear mappings of
finite rank, see \cite[Theorem 2.10]{BB18}. \newline
$(ii)$ The positive ideal property: see \cite[Proposition 2.3]{BB18}.\newline
The rest is obvious. In conclusion, the class $(\mathcal{D}%
_{p}^{m+},d_{p}^{m+}(.))$ is indeed a positive Banach left multi-ideal.%
\newline
\end{proof}

\textbf{The composition method. }Let $\mathcal{B}_{L}^{+}$ be a positive
left ideal. Let $X_{j}$ be Banach spaces with $1\leq j\leq m$, and let $F$
be a Banach lattice. A multilinear operator $T\in \mathcal{L}%
(X_{1},...,X_{m};F)$ belongs to $\mathcal{B}_{L}^{+}\circ \mathcal{L}$ if
there is a Banach space $Y$, a multilinear operator $S\in \mathcal{L}%
(X_{1},...,X_{m};Y)$, and an operator $u\in \mathcal{B}_{L}^{+}(Y;F)$ such
that 
\begin{equation*}
\begin{array}{ccc}
X_{1}\times ...\times X_{m} & \overset{T}{\longrightarrow } & F \\ 
S\downarrow & \nearrow u &  \\ 
Y &  & 
\end{array}%
\end{equation*}%
i.e., $T=u\circ S$. In this case, we denote $T\in \mathcal{B}_{L}^{+}\circ 
\mathcal{L}(X_{1},...,X_{m};F).$

\begin{remark}
An argument similar to \cite[Proposition 3.3]{BPR07}, the class $\mathcal{B}%
_{L}^{+}\circ \mathcal{L}$ is indeed a positive left multi-ideal.
\end{remark}

We have the following result.

\begin{proposition}
Let $\mathcal{B}_{L}^{+}$ be a positive left ideal. Let $X_{1},...,X_{m}$ be
Banach spaces and $F$ be a Banach lattice. For $T\in \mathcal{L}%
(X_{1},...,X_{m};F),$ the following statements are equivalent:\newline
a) The operator $T$ belongs to $\mathcal{B}_{L}^{+}\circ \mathcal{L}%
(X_{1},...,X_{m};F)$.\newline
b) The linearization $T_{L}$ belongs to $\mathcal{B}_{L}^{+}(X_{1}\widehat{%
\otimes }_{\pi }...\widehat{\otimes }_{\pi }X_{m};F).$
\end{proposition}

\begin{proof}
A demonstration analogous to Proposition 3.2. \cite{BPR07}.
\end{proof}

Similarly to \cite[Proposition 3.7 (a)]{BPR07}, if $\mathcal{B}_{L}^{+}$ is
a positive Banach left ideal, the composition $\mathcal{B}_{L}^{+}\circ 
\mathcal{L}$ forms a positive Banach left multi-ideal and we have%
\begin{equation*}
\left\Vert T\right\Vert =\inf \left\{ \left\Vert u\right\Vert _{\mathcal{B}%
_{L}^{+}}\left\Vert S\right\Vert :T=u\circ S\right\} .
\end{equation*}%
Proposition 3.1 in \cite{BB18} states that $\mathcal{D}_{p}^{m+}=\mathcal{D}%
_{p}^{+}\circ \mathcal{L}$. Consequently, $\mathcal{D}_{p}^{m+}$ represents
the positive Banach left multi-ideal generated by the composition method
from the positive Banach left ideal $\mathcal{D}_{p}^{+}.$

\subsection{\textsc{Positive right multi-ideal}}

Let us introduce the concept of \textit{positive right multi-ideals},
expanding upon the concept of multi-ideals. We also introduce the
factorization method that allows us to construct a positive right
multi-ideal from a given positive right ideal.

\begin{definition}
A positive \textit{right} multi-ideal (or positive \textit{right} ideal of
multilinear operators), denoted by $\mathcal{M}_{R}^{+},$ is a subclass of
all continuous multilinear operators of Banach lattices into a Banach space.
It fulfils the property: for all $m\in \mathbb{N}^{\ast }$, Banach lattices $%
E_{1},\ldots ,E_{m},$ and Banach space $X,$ the components%
\begin{equation*}
\mathcal{M}_{R}^{+}(E_{1},...,E_{m};X):=\mathcal{L}(E_{1},...,E_{m};X)\cap 
\mathcal{M}_{R}^{+}
\end{equation*}%
satisfy:\newline
$(i)$ $\mathcal{M}_{R}^{+}(E_{1},...,E_{m};X)$ is a linear subspace of $%
\mathcal{L}(E_{1},...,E_{m};X)$ which contains the $m$-linear mappings of
finite rank.\newline
$(ii)$ The positive ideal property: If $T\in \mathcal{M}_{R}^{+}\left(
E_{1},\ldots ,E_{m};X\right) ,u_{j}\in \mathcal{L}^{+}\left(
G_{j};E_{j}\right) $ for $j=$ $1,\ldots ,m$ and $v\in \mathcal{L}(X;Y)$,
then $v\circ T\circ \left( u_{1},\ldots ,u_{m}\right) $ is in $\mathcal{M}%
_{R}^{+}\left( G_{1},\ldots ,G_{m};Y\right) $.\newline
If $\Vert \cdot \Vert _{\mathcal{M}_{L}^{+}}:\mathcal{M}_{L}^{+}\rightarrow 
\mathbb{R}^{+}$ satisfies:\newline
a) $\left( \mathcal{M}_{R}^{+}(E_{1},...,E_{m};X),\Vert \cdot \Vert _{%
\mathcal{M}_{R}^{+}}\right) $ is a Banach (quasi-Banach) space for all
Banach lattices $E_{1},\ldots ,E_{m}$ and Banach space $X$.\newline
b) The $m$ -linear form $T^{m}:\mathbb{K}^{m}\rightarrow \mathbb{K}$ given
by $T^{m}\left( \lambda ^{1},\ldots ,\lambda ^{m}\right) =\lambda ^{1}\ldots
\lambda ^{m}$ satisfies $\left\Vert T^{m}\right\Vert _{\mathcal{M}%
_{R}^{+}}=1 $ for all $m$,\newline
c) $T\in \mathcal{M}_{R}^{+}\left( E_{1},\ldots ,E_{m};X\right) ,u_{j}\in 
\mathcal{L}^{+}\left( G_{j},E_{j}\right) $ for $j=$ $1,\ldots ,m$ and $v\in 
\mathcal{L}(X,Y)$ then 
\begin{equation*}
\left\Vert v\circ T\circ \left( u_{1},\ldots ,u_{m}\right) \right\Vert _{%
\mathcal{M}_{R}^{+}}\leq \Vert v\Vert \Vert T\Vert _{\mathcal{M}%
_{R}^{+}}\left\Vert u_{1}\right\Vert \ldots \left\Vert u_{m}\right\Vert .
\end{equation*}%
The class $\left( \mathcal{M}_{R}^{+},\Vert \cdot \Vert _{\mathcal{M}%
_{R}^{+}}\right) $ is referred to as a Banach (quasi-Banach) positive 
\textit{right} multi-ideal. When $m=1$, it specifically corresponds to the
case of a positive \textit{right} ideal.
\end{definition}

\begin{remark}
\label{Remark3}Condition $\left( ii\right) $ is equivalent to the following
assertion: For any $T\in \mathcal{M}_{R}^{+}\left( E_{1},\ldots
,E_{m};X\right) $, $u_{j}\in \mathcal{L}^{r}\left( G_{j};E_{j}\right) $ for $%
j=$ $1,\ldots ,m$ and $v\in \mathcal{L}(X;Y)$, the composition $v\circ
T\circ \left( u_{1},\ldots ,u_{m}\right) $ belongs to $\mathcal{M}%
_{R}^{+}\left( G_{1},\ldots ,G_{m};Y\right) $. In fact, we see this
equivalence without loss of generality for the case $m=2.$ Consider $%
u_{j}\in \mathcal{L}^{r}\left( G_{j};E_{j}\right) $ such that $%
u_{j}=u_{j}^{1}-u_{j}^{2}$ where $u_{j}^{1},u_{j}^{2}\subset \mathcal{L}%
^{+}\left( G_{j};E_{j}\right) $ $\left( j=1,2\right) $. Now, we can write%
\begin{eqnarray*}
&&v\circ T\circ \left( u_{1},u_{2}\right) \\
&=&v\circ T\circ \left( u_{1}^{1}-u_{1}^{2},u_{2}^{1}-u_{2}^{2}\right) \\
&=&v\circ T\circ \left( u_{1}^{1},u_{2}^{1}\right) -v\circ T\circ \left(
u_{1}^{1},u_{2}^{2}\right) -v\circ T\circ \left( u_{1}^{2},u_{2}^{1}\right)
+v\circ T\circ \left( u_{1}^{2},u_{2}^{2}\right) .
\end{eqnarray*}%
This is because the operators $v\circ T\circ \left(
u_{1}^{j},u_{2}^{k}\right) $ belong to $\mathcal{M}_{R}^{+}\left(
G_{1},G_{2};Y\right) $ $\left( j,k=1,2\right) $. Due to the linearity of $%
\mathcal{M}_{L}^{+}\left( G_{1},G_{2};Y\right) $, we conclude that $v\circ
T\circ \left( u_{1},u_{2}\right) \in \mathcal{M}_{R}^{+}\left(
G_{1},G_{2};Y\right) $. The converse is immediate.\newline
\end{remark}

\begin{remark}
It is obvious that all multi-ideals are actually positive \textit{right}
multi-ideals.
\end{remark}

\textbf{The factorization method. }For $m\in \mathbb{N}^{\ast },$ let $(%
\mathcal{B}_{j,R}^{+},\Vert \cdot \Vert _{\mathcal{B}_{j,R}^{+}})$ be Banach
positive right ideals for $i=1,...,m$. We define the class $\mathcal{L}(%
\mathcal{B}_{1,R}^{+},...,\mathcal{B}_{m,R}^{+})$ as follows: Let $%
E_{1},...,E_{m}$ be Banach lattices and $Y$ be a Banach space. An operator $%
T $ belongs to $\mathcal{L}(\mathcal{B}_{1,R}^{+},...,\mathcal{B}%
_{m,R}^{+})(E_{1},...,E_{m};Y)$ if there exist Banach spaces $%
X_{1},...,X_{m},$ operators $u_{j}\in \mathcal{B}_{j,R}^{+}(E_{j};X_{j})$ $%
\left( 1\leq j\leq m\right) ,$ and a multilinear operator $S\in \mathcal{L}%
(X_{1},...,X_{m};Y)$ such that%
\begin{equation*}
\begin{array}{ccccc}
E_{1} & \times ...\times & E_{m} & \overset{T}{\longrightarrow } & Y \\ 
u_{1}\downarrow &  & u_{m}\downarrow & \nearrow S &  \\ 
X_{1} & \times ...\times & X_{m} &  & 
\end{array}%
\end{equation*}%
i.e., $T=S\circ (u_{1},...,u_{m})$. In this case, we define the quasi-norm
of $T$ with respect to $\mathcal{L}(\mathcal{B}_{1,R}^{+},...,\mathcal{B}%
_{m,R}^{+})$ as%
\begin{equation*}
\Vert T\Vert _{\mathcal{L}(\mathcal{B}_{1,R}^{+},...,\mathcal{B}%
_{m,R}^{+})}=\inf \{\Vert S\Vert \prod_{j=1}^{m}\Vert u_{j}\Vert _{\mathcal{B%
}_{j,R}^{+}}\},
\end{equation*}%
where the infimum is taken over all possible factorizations of $T$ as
described above.

\begin{remark}
By a similar argument as in \cite[Theorem 1.4.1]{PelDoctorat}, we can show
that the class $\left( \mathcal{L}(\mathcal{B}_{1,R}^{+},...,\mathcal{B}%
_{m,R}^{+}),\Vert .\Vert _{\mathcal{L}(\mathcal{B}_{1,R}^{+},...,\mathcal{B}%
_{m,R}^{+})}\right) $ is a positive quasi-Banach \textit{right} multi-ideal.
\end{remark}

\subsection{\textsc{Positive multi-ideals}}

In this subsection, we introduce the definition of positive multi-ideals.
The construction of these positive multi-ideals is based on techniques
inspired by \cite{Pml}, utilizing multilinear operators defined exclusively
between Banach lattices.

\begin{definition}
A positive multi-ideal (or positive ideal of multilinear operators) is a
subclass $\mathcal{M}^{+}$ of all continuous multilinear operators between
Banach lattices. It is characterized by the property that for all $m\in 
\mathbb{N}^{\ast }$ and Banach lattices $E_{1},\ldots ,E_{m}$ and $F$, the
components 
\begin{equation*}
\mathcal{M}^{+}(E_{1},...,E_{m};F):=\mathcal{L}(E_{1},...,E_{m};F)\cap 
\mathcal{M}^{+}
\end{equation*}%
satisfy:\newline
$(i)$ $\mathcal{M}^{+}(E_{1},...,E_{m};F)$ is a linear subspace of $\mathcal{%
L}(E_{1},...,E_{m};F)$ which contains the $m$-linear mappings of finite rank.%
\newline
$(ii)$ The positive ideal property: If $T\in \mathcal{M}^{+}\left(
E_{1},\ldots ,E_{m};F\right) ,u_{j}\in \mathcal{L}^{+}\left(
G_{j};E_{j}\right) $ for $j=$ $1,\ldots ,m$ and $v\in \mathcal{L}^{+}(F;G)$,
then $v\circ T\circ \left( u_{1},\ldots ,u_{m}\right) $ is in $\mathcal{M}%
^{+}\left( G_{1},\ldots ,G_{m};G\right) $.\newline
If $\Vert \cdot \Vert _{\mathcal{M}^{+}}:\mathcal{M}^{+}\rightarrow \mathbb{R%
}^{+}$ satisfies:\newline
a) $\left( \mathcal{M}^{+}(E_{1},...,E_{m};F),\Vert \cdot \Vert _{\mathcal{M}%
^{+}}\right) $ is a Banach (quasi-Banach) space for all Banach lattices $%
E_{1},\ldots ,E_{m},$ $F$.\newline
b) The $m$ -linear form $T^{m}:\mathbb{K}^{m}\rightarrow \mathbb{K}$ given
by $T^{m}\left( \lambda ^{1},\ldots ,\lambda ^{m}\right) =\lambda ^{1}\ldots
\lambda ^{m}$ satisfies $\left\Vert T^{m}\right\Vert _{\mathcal{M}^{+}}=1$
for all $m$,\newline
c) $T\in \mathcal{M}^{+}\left( E_{1},\ldots ,E_{m};F\right) ,u_{j}\in 
\mathcal{L}^{+}\left( G_{j};E_{j}\right) $ for $j=$ $1,\ldots ,m$ and $v\in 
\mathcal{L}^{+}(F;G)$ then 
\begin{equation*}
\left\Vert v\circ T\circ \left( u_{1},\ldots ,u_{m}\right) \right\Vert _{%
\mathcal{M}^{+}}\leq \Vert v\Vert \Vert T\Vert _{\mathcal{M}^{+}}\left\Vert
u_{1}\right\Vert \ldots \left\Vert u_{m}\right\Vert .
\end{equation*}%
The class $\left( \mathcal{M}^{+},\Vert \cdot \Vert _{\mathcal{M}}\right) $
is referred to as a positive Banach (quasi-Banach) multi-ideal. In
particular, when $m=1$, we specifically refer to it as a positive Banach
(quasi-Banach) ideal.
\end{definition}

\begin{remark}
As mentioned in the remarks \ref{Remark2} and \ref{Remark3}$,$ we can
substitute the set $\mathcal{L}^{+}\left( G_{j};E_{j}\right) $ with the
space $\mathcal{L}^{r}\left( G_{j};E_{j}\right) $ $\left( 1\leq j\leq
m\right) $ and $\mathcal{L}^{+}(F;G)$ with the space $\mathcal{L}^{r}(F;G).$
Then, the condition $\left( ii\right) $ remains valid.
\end{remark}

\begin{remark}
1) It is evident that all multi-ideals are indeed positive multi-ideals.

2) Every positive right or left multi-ideal is positive multi-ideal.
\end{remark}

Let $\mathcal{B}_{L}^{+}$ be a positive left ideal, and $\mathcal{M}_{R}^{+}$
a positive right multi-ideal. The composition $\mathcal{B}_{L}^{+}\circ 
\mathcal{M}_{R}^{+}$ is defined as the class of multilinear operators $T$
that can be factorized by a Banach space as $T=v\circ S$. In other words,
for any Banach lattice $E_{1},...,E_{m}$ and $F$, and for $T\in \mathcal{B}%
_{L}^{+}\circ \mathcal{M}_{R}^{+}(E_{1},...,E_{m};F)$, there exist a Banach
space $X$, an element $v$ that belongs to $\mathcal{B}_{L}^{+}\left(
X,E\right) $, and a multilinear operator $S$ that belongs to $\mathcal{M}%
_{R}^{+}(E_{1},...,E_{m};X)$ so that%
\begin{equation*}
\begin{array}{ccc}
E_{1}\times ...\times E_{m} & \overset{T}{\longrightarrow } & F \\ 
S\downarrow & \nearrow v &  \\ 
X &  & 
\end{array}%
\end{equation*}%
In other words, $T$ can be expressed as $T=v\circ S$. If $\mathcal{B}%
_{L}^{+} $ is positive Banach left ideal and $\mathcal{M}_{R}^{+}$ is
positive Banach right multi-ideal and $T\in \mathcal{B}_{L}^{+}\circ 
\mathcal{M}_{R}^{+}(E_{1},...,E_{m};F)$ we define%
\begin{equation*}
\left\Vert T\right\Vert _{\mathcal{B}_{L}^{+}\circ \mathcal{M}_{R}^{+}}=\inf
\left\{ \left\Vert v\right\Vert _{\mathcal{B}_{L}^{+}}\left\Vert
S\right\Vert _{\mathcal{M}_{R}^{+}}:T=v\circ S\right\} .
\end{equation*}

A similar argument to that used in the Proposition \ref{Propo1} and Theorem %
\ref{Theorem1} can be applied to establish the following result.

\begin{theorem}
If $\mathcal{B}_{L}^{+}$ is a positive Banach left ideal and $\mathcal{M}%
_{R}^{+}$ is a positive Banach right multi-ideal, then%
\begin{equation*}
\left( \mathcal{B}_{L}^{+}\circ \mathcal{M}_{R}^{+},\left\Vert .\right\Vert
_{\mathcal{B}_{L}^{+}\circ \mathcal{M}_{R}^{+}}\right)
\end{equation*}%
is a positive quasi-Banach multi-ideal.
\end{theorem}

Let $\mathcal{B}_{1,R}^{+},...,\mathcal{B}_{m,R}^{+}$ be positive right
ideals, and let $\mathcal{M}_{L}^{+}$ be a positive left multi-ideal. The
class $\mathcal{M}_{L}^{+}\left( \mathcal{B}_{1,R}^{+},...,\mathcal{B}%
_{m,R}^{+}\right) $ is defined as the set of multilinear operators $T$ that
can be expressed as $T=S\left( v_{1},...,v_{m}\right) $. Specifically, for
any Banach lattices $E_{1},...,E_{m}$ and $F$, we say that $T$ belongs to $%
\mathcal{M}_{L}^{+}\left( \mathcal{B}_{1,R}^{+},...,\mathcal{B}%
_{m,R}^{+}\right) (E_{1},...,E_{m};F)$ if there exist Banach spaces $%
X_{1},...,X_{m}$, elements $u_{j}$ belonging to $\mathcal{B}_{j,R}^{+}\left(
E_{j},X_{j}\right) \left( 1\leq j\leq m\right) $, and a multilinear operator 
$S$ belonging to $\mathcal{M}_{L}^{+}(X_{1},...,X_{m};F)$ such that%
\begin{equation*}
\begin{array}{ccccc}
E_{1} & \times ...\times & E_{m} & \overset{T}{\longrightarrow } & F \\ 
u_{1}\downarrow &  & u_{m}\downarrow & \nearrow S &  \\ 
X_{1} & \times ...\times & X_{m} &  & 
\end{array}%
\end{equation*}%
In other words, $T$ can be represented as $T=S(u_{1},...,u_{m})$. If $%
\mathcal{B}_{j,R}^{+}$ $\left( 1\leq j\leq m\right) $\ are positive Banach
right ideals and $\mathcal{M}_{L}^{+}$ is positive Banach left multi-ideal
and if $T\in \mathcal{M}_{L}^{+}\left( \mathcal{B}_{1,R}^{+},...,\mathcal{B}%
_{m,R}^{+}\right) (E_{1},...,E_{m};F),$ we define%
\begin{equation*}
\left\Vert T\right\Vert _{\mathcal{M}_{L}^{+}\left( \mathcal{B}%
_{1,R}^{+},...,\mathcal{B}_{m,R}^{+}\right) }=\inf \left\{
\dprod\limits_{i=1}^{m}\left\Vert u_{i}\right\Vert _{\mathcal{B}%
_{i,R}^{+}}\left\Vert S\right\Vert _{\mathcal{M}%
_{L}^{+}}:T=S(u_{1},...,u_{m})\right\} .
\end{equation*}

According to \cite[Theorem 1]{Pml}, we present the following theorem.

\begin{theorem}
If $\mathcal{B}_{j,R}^{+}$ $\left( 1\leq j\leq m\right) $ are positive
Banach right ideals and $\mathcal{M}_{L}^{+}$ is a positive Banach left
multi-ideal, then%
\begin{equation*}
\left( \mathcal{M}_{L}^{+}\left( \mathcal{B}_{1,R}^{+},...,\mathcal{B}%
_{m,R}^{+}\right) ,\left\Vert .\right\Vert _{\mathcal{M}_{L}^{+}\left( 
\mathcal{B}_{1,R}^{+},...,\mathcal{B}_{m,R}^{+}\right) }\right)
\end{equation*}%
is a positive quasi-Banach multi-ideal.
\end{theorem}

In \cite[Theorem 3.2]{BBH21}, the authors established the following
factorization 
\begin{equation*}
\mathcal{N}_{p}^{m+}=\mathcal{D}_{p}^{m+}(\Pi _{p}^{+},...,\Pi _{p}^{+}).
\end{equation*}%
In other words, the class $\left( \mathcal{N}_{p}^{m+},\eta
_{p}^{m+}(.)\right) $ represents the positive Banach multi-ideal of type $%
\mathcal{M}_{L}^{+}\left( \mathcal{B}_{1,R}^{+},...,\mathcal{B}%
_{m,R}^{+}\right) $ where $\mathcal{M}_{L}^{+}=\mathcal{D}_{p}^{m+}$ and $%
\mathcal{B}_{j,R}^{+}=\Pi _{p}^{+}$ for $1\leq j\leq m.$ On the other hand,
the class $\left( \mathcal{N}_{p}^{m+},\left\Vert .\right\Vert _{\mathcal{D}%
_{p}^{m+}\left( \Pi _{p}^{+},...,\Pi _{p}^{+}\right) }\right) $ is a
positive quasi Banach multi-ideal. Since $\mathcal{D}_{p}^{m+}=\mathcal{D}%
_{p}^{+}\circ \mathcal{L}$, we have%
\begin{equation*}
\mathcal{N}_{p}^{m+}=\mathcal{D}_{p}^{+}\circ \mathcal{L}(\Pi
_{p}^{+},...,\Pi _{p}^{+}).
\end{equation*}%
This implies that the class $\mathcal{N}_{p}^{m+}$ consists of the
multilinear operators that can be obtained by composing positive strongly $p$%
-summing operators with multilinear operators derived via a factorization
method using the positive right ideal $\Pi _{p}^{+}$.

\section{\textsc{Positive (p}$_{1}$\textsc{,...,p}$_{m}$\textsc{%
;r)-dominated multilinear operators}}

The concept of $(p_{1},\ldots ,p_{m};r)$-dominated multilinear operators was
introduced by Achour \cite{Ach11}. This notion is a natural generalization
of the concept of $(p,q)$-dominated linear operators originally studied by
Pietsch in \cite{PIETSCHoi}. In this section we study the positive
multilinear version of this concept and give a good example of a positive
multi-ideal.

\begin{definition}
\label{definition7}Consider $1\leq r,p,p_{1},\ldots ,p_{m}\leq \infty $ such
that $\frac{1}{p}=\frac{1}{p_{1}}+\ldots +\frac{1}{p_{m}}+\frac{1}{r}$. Let $%
E_{1},...,E_{m}$ and $F$ be Banach lattices. A mapping $T\in \mathcal{L}%
\left( E_{1},\ldots ,E_{m};F\right) $ is said to be positive $(p_{1},\ldots
,p_{m};r)$-dominated if there is a constant $C>0$ such that for every $%
(x_{i}^{1},...,x_{i}^{m})\in E_{1}^{+}\times ...\times E_{m}^{+}$ $\left(
1\leq i\leq n\right) $ and $y_{1}^{\ast },\ldots ,y_{n}^{\ast }\in F^{\ast
+} $, the following inequality holds: 
\begin{equation}
\left\Vert \left( \left\langle T\left( x_{i}^{1},\ldots ,x_{i}^{m}\right)
,y_{i}^{\ast }\right\rangle \right) _{i=1}^{n}\right\Vert _{p}\leq
C\prod_{j=1}^{m}\left\Vert \left( x_{i}^{j}\right) _{i=1}^{n}\right\Vert
_{p_{j},w}\left\Vert \left( y_{i}^{\ast }\right) _{i=1}^{n}\right\Vert
_{r,w}.  \label{def1sec3}
\end{equation}%
The space consisting of all such mappings is denoted by $\mathcal{D}_{\left(
p_{1},\ldots ,p_{m};r\right) }^{+}\left( E_{1},\ldots ,E_{m};F\right) $. In
this case, we define 
\begin{equation*}
d_{\left( p_{1},\ldots ,p_{m};r\right) }^{+}(T)=\inf \{C>0:C\text{ \textit{%
satisfies inequality} }(\ref{def1sec3})\}.
\end{equation*}
\end{definition}

It is easy to check that every $(p_{1},\ldots ,p_{m};r)$-dominated
multilinear operator is positive $(p_{1},\ldots ,p_{m};r)$-dominated. Then
we have through \cite[Proposition 2.4 (i)]{Ach11} we have%
\begin{equation*}
\mathcal{L}_{f}(X_{1},...,X_{m};F)\subset \mathcal{D}_{\left( p_{1},\ldots
,p_{m};r\right) }^{+}\left( E_{1},\ldots ,E_{m};F\right) .
\end{equation*}

In the next result, we give the following equivalent definition.

\begin{theorem}
\label{Theorem equivalent1}Let $1\leq r,p,p_{1},\ldots ,p_{m}\leq \infty $
with $\frac{1}{p}=\frac{1}{p_{1}}+\ldots +\frac{1}{p_{m}}+\frac{1}{r}$ and $%
T\in \mathcal{L}\left( E_{1},\ldots ,E_{m};F\right) $. The following
properties are equivalent:\newline
$(a)$ The operator $T$ is positive $(p_{1},\ldots ,p_{m};r)$-dominated.%
\newline
$(b)$ There is a constant $C>0$ such that for any $(x_{i}^{1},...,x_{i}^{m})%
\in E_{1}\times ...\times E_{m}$ $\left( 1\leq i\leq n\right) $ and $%
y_{1}^{\ast },\ldots ,y_{n}^{\ast }\in F^{\ast }$, we have 
\begin{equation}
\left\Vert \left( \left\langle T\left( x_{i}^{1},\ldots ,x_{i}^{m}\right)
,y_{i}^{\ast }\right\rangle \right) _{i=1}^{n}\right\Vert _{p}\leq
C\prod_{j=1}^{m}\left\Vert \left( x_{i}^{j}\right) _{i=1}^{n}\right\Vert
_{p_{j},\left\vert w\right\vert }\left\Vert \left( y_{i}^{\ast }\right)
_{i=1}^{n}\right\Vert _{r,\left\vert w\right\vert }.  \label{def2sec3}
\end{equation}%
In this case, we define 
\begin{equation*}
d_{\left( p_{1},\ldots ,p_{m};r\right) }^{+}(T)=\inf \{C>0:C\text{ \textit{%
satisfies inequality} }(\ref{def2sec3})\}.
\end{equation*}
\end{theorem}

\begin{proof}
$(b)\Rightarrow (a):$ Immediately applying Definition \ref{definition7} for $%
(x_{i}^{1},...,x_{i}^{m})\in E_{1}^{+}\times ...\times E_{m}^{+}$, $1\leq
i\leq n$ and $y_{1}^{\ast },...,y_{n}^{\ast }\in F^{\ast +}$.

$(a)\Rightarrow (b):$ Suppose that $T$ is positive $(p_{1},\ldots ,p_{m};r)$%
-dominated. For convenience, we prove only the inequality for the case when $%
m=2$. Let $\left( x_{i}^{1},x_{i}^{2}\right) \in E_{1}\times E_{2},(1\leq
i\leq n)$ $y_{1}^{\ast },\ldots ,y_{n}^{\ast }\in F^{\ast }$, then one has%
\begin{eqnarray*}
&&(\sum_{i=1}^{n}\left\vert \left\langle T(x_{i}^{1},x_{i}^{2}),y_{i}^{\ast
}\right\rangle \right\vert ^{p})^{\frac{1}{p}}=(\sum_{i=1}^{n}\left\vert
\left\langle T\left( x_{i}^{1+}-x_{i}^{1-},x_{i}^{2+}-x_{i}^{2-}\right)
,y_{i}^{\ast }\right\rangle \right\vert ^{p})^{\frac{1}{p}} \\
&\leq &(\sum_{i=1}^{n}\left\vert \left\langle T\left(
x_{i}^{1+},x_{i}^{2+}\right) ,y_{i}^{\ast }\right\rangle \right\vert ^{p})^{%
\frac{1}{p}}+(\sum_{i=1}^{n}\left\vert \left\langle T\left(
x_{i}^{1+},x_{i}^{2-}\right) ,y_{i}^{\ast }\right\rangle \right\vert ^{p})^{%
\frac{1}{p}}+ \\
&&(\sum_{i=1}^{n}\left\vert \left\langle T\left(
x_{i}^{1-},x_{i}^{2+}\right) ,y_{i}^{\ast }\right\rangle \right\vert ^{p})^{%
\frac{1}{p}}+(\sum_{i=1}^{n}\left\vert \left\langle T\left(
x_{i}^{1-},x_{i}^{2-}\right) ,y_{i}^{\ast }\right\rangle \right\vert ^{p})^{%
\frac{1}{p}}
\end{eqnarray*}%
which is less than or equal to%
\begin{eqnarray*}
&\leq &(\sum_{i=1}^{n}\left\vert \left\langle T\left(
x_{i}^{1+},x_{i}^{2+}\right) ,y_{i}^{\ast +}\right\rangle \right\vert ^{p})^{%
\frac{1}{p}}+(\sum_{i=1}^{n}\left\vert \left\langle T\left(
x_{i}^{1+},x_{i}^{2+}\right) ,y_{i}^{\ast -}\right\rangle \right\vert ^{p})^{%
\frac{1}{p}}+ \\
&&(\sum_{i=1}^{n}\left\vert \left\langle T\left(
x_{i}^{1+},x_{i}^{2-}\right) ,y_{i}^{\ast +}\right\rangle \right\vert ^{p})^{%
\frac{1}{p}}+(\sum_{i=1}^{n}\left\vert \left\langle T\left(
x_{i}^{1+},x_{i}^{2-}\right) ,y_{i}^{\ast -}\right\rangle \right\vert ^{p})^{%
\frac{1}{p}}+ \\
&&(\sum_{i=1}^{n}\left\vert \left\langle T\left(
x_{i}^{1-},x_{i}^{2+}\right) ,y_{i}^{\ast +}\right\rangle \right\vert ^{p})^{%
\frac{1}{p}}+(\sum_{i=1}^{n}\left\vert \left\langle T\left(
x_{i}^{1-},x_{i}^{2+}\right) ,y_{i}^{\ast -}\right\rangle \right\vert ^{p})^{%
\frac{1}{p}}+ \\
&&(\sum_{i=1}^{n}\left\vert \left\langle T\left(
x_{i}^{1-},x_{i}^{2+}\right) ,y_{i}^{\ast +}\right\rangle \right\vert ^{p})^{%
\frac{1}{p}}+(\sum_{i=1}^{n}\left\vert \left\langle T\left(
x_{i}^{1-},x_{i}^{2+}\right) ,y_{i}^{\ast -}\right\rangle \right\vert ^{p})^{%
\frac{1}{p}}+ \\
&&(\sum_{i=1}^{n}\left\vert \left\langle T\left(
x_{i}^{1-},x_{i}^{2-}\right) ,y_{i}^{\ast +}\right\rangle \right\vert ^{p})^{%
\frac{1}{p}}+(\sum_{i=1}^{n}\left\vert \left\langle T\left(
x_{i}^{1-},x_{i}^{2-}\right) ,y_{i}^{\ast -}\right\rangle \right\vert ^{p})^{%
\frac{1}{p}},
\end{eqnarray*}%
finally we have%
\begin{equation*}
(\sum_{i=1}^{n}\left\vert \left\langle T(x_{i}^{1},x_{i}^{2}),y_{i}^{\ast
}\right\rangle \right\vert ^{p})^{\frac{1}{p}}\leq 8d_{\left(
p_{1},p_{2};r\right) }^{+}(T)\Vert (x_{i}^{1})_{i=1}^{n}\Vert
_{p_{1},|w|}\Vert (x_{i}^{2})_{i=1}^{n}\Vert _{p_{2},|w|}\Vert (y_{i}^{\ast
})_{i=1}^{n}\Vert _{r,|w|}.
\end{equation*}
\end{proof}

\begin{proposition}
The class $(\mathcal{D}_{\left( p_{1},...,p_{m};r\right) }^{+},d_{\left(
p_{1},...,p_{m};r\right) }^{+})$ is a Banach positive multi-ideal.
\end{proposition}

\begin{proof}
We will verify the positive ideal property; the proof of the rest is
straightforward. Let $E_{1},\ldots ,E_{m}$ and $F$ are Banach lattices. Let $%
T\in \mathcal{D}_{\left( p_{1},...,p_{m};r\right) }^{+}\left( E_{1},\ldots
,E_{m};F\right) ,$ $u_{j}\in \mathcal{L}^{+}\left( G_{j};E_{j}\right) $ $%
\left( 1\leq j\leq m\right) $ and $v\in \mathcal{L}^{+}(F;G)$ where $%
G_{1},\ldots ,G_{m}$ and $G$ are Banach lattices. Let $%
(x_{i}^{1},...,x_{i}^{m})\in G_{1}^{+}\times ...\times G_{m}^{+}$ $\left(
1\leq i\leq n\right) $ and $y_{1}^{\ast },\ldots ,y_{n}^{\ast }\in G^{\ast
+}.$ Since $T\in \mathcal{D}_{\left( p_{1},...,p_{m};r\right) }^{+}\left(
E_{1},\ldots ,E_{m};F\right) ,$ $u_{j}\left( x_{i}^{j}\right) \geq 0$ and $%
v^{\ast }\left( y_{i}^{\ast }\right) \geq 0$ $\left( 1\leq j\leq m,1\leq
i\leq n\right) $ we have%
\begin{eqnarray*}
&&\left\Vert \left( \left\langle v\circ T\circ \left( u_{1},\ldots
,u_{m}\right) \left( x_{i}^{1},\ldots ,x_{i}^{m}\right) ,y_{i}^{\ast
}\right\rangle \right) _{i=1}^{n}\right\Vert _{p} \\
&=&\left\Vert \left( \left\langle T\left( u_{1}\left( x_{i}^{1}\right)
,\ldots ,u_{m}\left( x_{i}^{m}\right) \right) ,v^{\ast }\left( y_{i}^{\ast
}\right) \right\rangle \right) _{i=1}^{n}\right\Vert _{p} \\
&\leq &d_{\left( p_{1},\ldots ,p_{m};r\right)
}^{+}(T)\prod_{j=1}^{m}\left\Vert \left( u_{j}\left( x_{i}^{j}\right)
\right) _{i=1}^{n}\right\Vert _{p_{j},w}\left\Vert \left( v^{\ast }\left(
y_{i}^{\ast }\right) \right) _{i=1}^{n}\right\Vert _{r,w} \\
&\leq &d_{\left( p_{1},\ldots ,p_{m};r\right) }^{+}(T)\left\Vert
u_{1}\right\Vert \ldots \left\Vert u_{m}\right\Vert \Vert v\Vert
\prod_{j=1}^{m}\left\Vert \left( x_{i}^{j}\right) _{i=1}^{n}\right\Vert
_{p_{i},w}\left\Vert \left( y_{i}^{\ast }\right) _{i=1}^{n}\right\Vert _{r,w}
\end{eqnarray*}%
thus $v\circ T\circ \left( u_{1},\ldots ,u_{m}\right) $ is in $\mathcal{D}%
_{\left( p_{1},...,p_{m};r\right) }^{+}\left( G_{1},\ldots ,G_{m};G\right) $
and we have%
\begin{equation*}
d_{\left( p_{1},\ldots ,p_{m};r\right) }^{+}\left( v\circ T\circ \left(
u_{1},\ldots ,u_{m}\right) \right) \leq d_{\left( p_{1},\ldots
,p_{m};r\right) }^{+}(T)\left\Vert u_{1}\right\Vert \ldots \left\Vert
u_{m}\right\Vert \Vert v\Vert .
\end{equation*}
\end{proof}

\begin{proposition}
\label{PropoCompo}Let $A$ be a Cohen positive strongly $r^{\ast }$-summing
multilinear operator and $u_{j}$ be positive $p_{j}$-summing linear
operators with $1\leq j\leq m$. Then $T=A\circ \left( u_{1},\ldots
,u_{m}\right) $ is positive $(p_{1},...,p_{m};r)$-dominated and we have%
\begin{equation*}
d_{\left( p_{1},\ldots ,p_{m};r\right) }^{+}(T)\leq d_{r^{\ast
}}^{m+}(A)\prod\limits_{j=1}^{m}\pi _{p_{j}}^{+}\left( u_{j}\right) .
\end{equation*}
\end{proposition}

\begin{proof}
By \cite[Theorem 2.5]{BB18}, there exists $\mu $ on $B_{F^{\ast \ast }}^{+}$
such that, for all $x^{j}\in E_{j}$ $(1\leq j\leq m)$ and $y^{\ast }\in
B_{F^{\ast }}^{+},\ $we have 
\begin{equation*}
\begin{tabular}{lll}
$\left\vert \left\langle T(x^{1},...,x^{m}),y^{\ast }\right\rangle
\right\vert $ & $=$ & $\left\vert \left\langle A\left( u_{1}\left(
x^{1}\right) ,\ldots ,u_{m}\left( x^{m}\right) \right) ,y^{\ast
}\right\rangle \right\vert $ \\ 
& $\leq $ & $d_{r^{\ast }}^{m+}(A)\prod\limits_{j=1}^{m}\left\Vert
u_{j}\left( x^{j}\right) \right\Vert \left( \int_{B_{F^{\ast \ast
}}^{+}}\left\vert \left\langle y^{\ast },y^{\ast \ast }\right\rangle
\right\vert ^{r}d\mu \right) ^{\frac{1}{r}}.$%
\end{tabular}%
\end{equation*}%
Since $u_{j}$ is positive $p_{j}$-summing then\textbf{, }by (\ref%
{DomiSumming}) there is a probability measure $\mu _{j}$ on $B_{E_{j}^{\ast
}}^{+}$ such that for all $x^{j}\in E_{j}^{+}$ 
\begin{equation*}
\left\Vert u_{j}\left( x^{j}\right) \right\Vert \leq \pi _{p_{j}}^{+}\left(
u_{j}\right) \left( \int_{B_{E_{j}^{\ast }}^{+}}\langle x^{j},x_{j}^{\ast
}\rangle ^{p_{j}}d\mu _{j}\right) ^{\frac{1}{p_{j}}}.
\end{equation*}%
Consequently%
\begin{eqnarray*}
&&\left\vert \left\langle T(x^{1},...,x^{m}),y^{\ast }\right\rangle
\right\vert \\
&\leq &d_{r^{\ast }}^{m+}(A)\dprod\limits_{j=1}^{m}\pi _{p_{j}}^{+}\left(
u_{j}\right) \left( \int_{B_{E_{j}^{\ast }}^{+}}\langle x^{j},x_{j}^{\ast
}\rangle ^{p_{j}}d\mu _{j}\right) ^{\frac{1}{p_{j}}}\left( \int_{B_{F^{\ast
\ast }}^{+}}\left\vert \left\langle y^{\ast },y^{\ast \ast }\right\rangle
\right\vert ^{r}d\mu \right) ^{\frac{1}{r}}
\end{eqnarray*}
Therefore, $T$ is positive $(p_{1},...,p_{m};r)$-dominated by Theorem \ref%
{thdo1} and 
\begin{equation*}
d_{\left( p_{1},\ldots ,p_{m};r\right) }^{+}(T)\leq d_{r^{\ast
}}^{m+}(A)\prod\limits_{j=1}^{m}\pi _{p_{j}}^{+}\left( u_{j}\right) .
\end{equation*}
\end{proof}

Now, we characterize the positive $(p_{1},\ldots ,p_{m};r)$-dominated
multilinear operators by the Pietsch domination theorem. For this purpose,
we use the full general Pietsch domination theorem given by Pellegrino et
al. in \cite[Theorem 4.6]{PSS12}.

\begin{theorem}[Pietsch domination theorem]
\label{thdo1} Let $1\leq r,p,p_{1},\ldots ,p_{m}\leq \infty $ with $\frac{1}{%
p}=\frac{1}{p_{1}}+\ldots +\frac{1}{p_{m}}+\frac{1}{r}$. Let $%
E_{1},...,E_{m} $ and $F$ be Banach lattices. The following statements are
equivalent:

1) The operator $T\in \mathcal{L}\left( E_{1},\ldots ,E_{m};F\right) $ is
positive $(p_{1},\ldots ,p_{m};r)$-dominated.

2) There is a constant $C>0$ and Borel probability measures $\mu _{j}$ on $%
B_{E_{j}^{\ast }}^{+}$ ($1\leq j\leq m$) and $\mu _{m+1}$ on $B_{F^{\ast
\ast }}^{+}$ such that 
\begin{equation}
\begin{array}{ll}
& |\langle T(x^{1},...,x^{m}),y^{\ast }\rangle | \\ 
& \leq C\prod_{j=1}^{m}\left( \int_{B_{E_{j}^{\ast }}^{+}}\langle
|x^{j}|,x_{j}^{\ast }\rangle ^{p_{j}}d\mu _{j}\right) ^{\frac{1}{p_{j}}%
}\left( \int_{B_{F^{\ast \ast }}^{+}}\langle |y^{\ast }|,y^{\ast \ast
}\rangle ^{r}d\mu _{m+1}\right) ^{\frac{1}{r}}%
\end{array}
\label{def2sec4}
\end{equation}%
for all $(x^{1},...,x^{m},y^{\ast })\in E_{1}\times ...\times E_{m}\times
F^{\ast }.$ Therefore, we have%
\begin{equation*}
d_{\left( p_{1},\ldots ,p_{m};r\right) }^{+}(T)=\inf \{C>0:C\text{ \textit{%
satisfies inequality} }(\ref{def2sec4})\}.
\end{equation*}

3) There is a constant $C>0$ and Borel probability measures $\mu _{j}$ on $%
B_{E_{j}^{\ast }}^{+}$ ($1\leq j\leq m$) and $\mu _{m+1}$ on $B_{F^{\ast
\ast }}^{+}$ such that%
\begin{equation}
\begin{array}{ll}
& |\langle T(x^{1},...,x^{m}),y^{\ast }\rangle | \\ 
& \leq C\prod_{j=1}^{m}\left( \int_{B_{E_{j}^{\ast }}^{+}}\langle
x^{j},x_{j}^{\ast }\rangle ^{p_{j}}d\mu _{j}\right) ^{\frac{1}{p_{j}}}\left(
\int_{B_{F^{\ast \ast }}^{+}}\langle y^{\ast },y^{\ast \ast }\rangle
^{r}d\mu _{m+1}\right) ^{\frac{1}{r}}%
\end{array}
\label{def2sec5}
\end{equation}%
for all $(x^{1},...,x^{m},y^{\ast })\in E_{1}^{+}\times ...\times
E_{m}^{+}\times F^{\ast +}$. Therefore, we have%
\begin{equation*}
d_{\left( p_{1},\ldots ,p_{m};r\right) }^{+}(T)=\inf \{C>0:C\text{ \textit{%
satisfies inequality} }(\ref{def2sec5})\}.
\end{equation*}
\end{theorem}

\begin{proof}
$1)\Leftrightarrow 2):$ Choosing the parameters%
\begin{equation*}
\left\{ 
\begin{array}{l}
K_{j}=B_{E_{j}^{\ast }}^{+},\text{ }j=1,\ldots ,m \\ 
K_{m+1}=B_{F^{\ast \ast }}^{+} \\ 
S\left( T,\lambda ,x^{1},\ldots ,x^{m},y^{\ast }\right) =|\langle
T(x^{1},...,x^{m}),y^{\ast }\rangle | \\ 
R_{j}(x_{j}^{\ast },\lambda ,x^{j})=\langle |x^{j}|,x_{j}^{\ast }\rangle
,j=1,\ldots ,m \\ 
R_{m+1}(y^{\ast \ast },\lambda ,y^{\ast })=\langle |y^{\ast }|,y^{\ast \ast
}\rangle .%
\end{array}%
\right.
\end{equation*}%
These maps satisfy the conditions $\left( 1\right) $ and $\left( 2\right) $
in \cite[Page 1255]{PSS12}. From this, we can easily conclude that $%
T:E_{1}\times ...\times E_{m}\rightarrow F$ is positive $(p_{1},\ldots
,p_{m};r)$-dominated if, and only if, 
\begin{equation*}
\begin{array}{lll}
S\left( T,\lambda ,x^{1},\ldots ,x^{m},y^{\ast }\right) & \leq & 
C\prod_{j=1}^{m}\left( \int_{K_{j}}R_{j}(\varphi _{j},\lambda
,x^{j})^{p_{j}}d\mu _{j}\right) ^{\frac{1}{p_{j}}} \\ 
&  & \times \left( \int_{K_{m+1}}R_{m+1}(\varphi _{m+1},\lambda ,y^{\ast
})^{r}\right) ^{\frac{1}{r}}.%
\end{array}%
\end{equation*}%
i.e., $T$ is $R_{1},...,R_{m+1}$-$S$-abstract $(p_{1},...,p_{m};r)$-summing.
Theorem \cite[Theorem 4.6]{PSS12} states that $T$ is $R_{1},...,R_{m+1}$-$S$%
-abstract $(p_{1},...,p_{m};r)$-summing if and only if, there exists a
positive constant $C$ and probability measures $\mu _{j}$ on $K_{j}$, $%
j=1,...,m+1,$ such that%
\begin{equation*}
\begin{array}{lll}
S\left( T,\lambda ,x^{1},\ldots ,x^{m},y^{\ast }\right) & \leq & 
C\prod_{j=1}^{m}\left( \int_{K_{j}}R_{j}(\varphi _{j},\lambda
,x^{j})^{p_{j}}d\mu _{j}\right) ^{\frac{1}{p_{j}}} \\ 
&  & \times \left( \int_{K_{m+1}}R_{m+1}(\varphi _{m+1},\lambda ,y^{\ast
})^{r}\right) ^{\frac{1}{r}}.%
\end{array}%
\end{equation*}%
Consequently 
\begin{equation}
\begin{array}{ll}
& |\langle T(x^{1},...,x^{m}),y^{\ast }\rangle | \\ 
& \leq C\prod_{j=1}^{m}\left( \int_{B_{E_{j}^{\ast }}^{+}}\langle
|x^{j}|,x_{j}^{\ast }\rangle ^{p_{j}}d\mu _{j}\right) ^{\frac{1}{p_{j}}%
}\left( \int_{B_{F^{\ast \ast }}^{+}}\langle |y^{\ast }|,y^{\ast \ast
}\rangle ^{r}d\mu _{m+1}\right) ^{\frac{1}{r}}.%
\end{array}%
\end{equation}

$2)\Leftrightarrow 3):$ Straightforward by using the idea of the proof of
Theorem \ref{Theorem equivalent1}.
\end{proof}

As an immediate consequence of Theorem \ref{thdo1}, we can show that if $%
p_{j}\leq q_{j}$ and $r\leq s$ then 
\begin{equation*}
\mathcal{D}_{\left( p_{1},...,p_{m};r\right) }^{+}\left( E_{1},\ldots
,E_{m};F\right) \subset \mathcal{D}_{\left( q_{1},...,q_{m};s\right)
}^{+}\left( E_{1},\ldots ,E_{m};F\right) .
\end{equation*}%
The following result demonstrates that the class of positive $\left(
p_{1},\ldots ,p_{m};r\right) $-dominated multilinear operators can be
construed as%
\begin{equation*}
\mathcal{D}_{\left( p_{1},\ldots ,p_{m};r\right) }^{+}=\mathcal{D}_{r^{\ast
}}^{m+}\left( \Pi _{p_{1}}^{+},\ldots ,\Pi _{p_{m}}^{+}\right) .
\end{equation*}%
This represents a positive variant of the Kwapie\'{n} factorization.

\begin{theorem}
\label{thfa} Let $1\leq r,p,p_{1},\ldots ,p_{m}\leq \infty $ with $\frac{1}{p%
}=\frac{1}{p_{1}}+\ldots +\frac{1}{p_{m}}+\frac{1}{r}$. Then, $T\in \mathcal{%
L}(E_{1},...,E_{m};F)$ is positive $\left( p_{1},\ldots ,p_{m};r\right) $%
-dominated if and only if there exist Banach spaces $X_{1},\ldots ,X_{m}$, a
Cohen positive strongly $r^{\ast }$-summing multilinear operator $%
A:X_{1}\times ...\times X_{m}\rightarrow F$ and linear operators $u_{j}\in
\Pi _{p_{j}}^{+}\left( E_{j};X_{j}\right) $ so that $T=A\circ \left(
u_{1},\ldots ,u_{m}\right) $, i.e., 
\begin{equation*}
\mathcal{D}_{\left( p_{1},\ldots ,p_{m};r\right) }^{+}(E_{1},...,E_{m};F)=%
\mathcal{D}_{r^{\ast }}^{m+}\left( \Pi _{p_{1}}^{+},\ldots ,\Pi
_{p_{m}}^{+}\right) (E_{1},...,E_{m};F).
\end{equation*}%
Moreover 
\begin{equation*}
d_{\left( p_{1},\ldots ,p_{m};r\right) }^{+}(T)=\inf \left\{ d_{r^{\ast
}}^{m+}(A)\prod_{j=1}^{m}\pi _{p_{j}}^{+}\left( u_{j}\right) :T=A\circ
\left( u_{1},\ldots ,u_{m}\right) \right\} .
\end{equation*}
\end{theorem}

\begin{proof}
Suppose that $T=A\circ \left( u_{1},\ldots ,u_{m}\right) $ where $u_{j}$ is
positive $p_{j}$-summing and $A$ is Cohen positive strongly $r^{\ast }$%
-summing multilinear operator. The result follows immediately from
Proposition \ref{PropoCompo}.

Conversely, let $T\in \mathcal{D}_{\left( p_{1},\ldots ,p_{m};r\right)
}^{+}(E_{1},...,E_{m};F).$ By Theorem \ref{thdo1}, there exist probability
measures $\mu _{j}$ on $K_{j}=B_{E_{j}^{\ast }}^{+}$ and $\mu $ on $%
B_{F^{\ast \ast }}^{+}$ such that for all $x^{j}\in E_{j}^{+}$ and $y^{\ast
}\in F^{\ast +}$ we have%
\begin{eqnarray*}
&&\left\vert \left\langle T(x^{1},...,x^{m}),y^{\ast }\right\rangle
\right\vert \\
&\leq &d_{\left( p_{1},\ldots ,p_{m};r\right)
}^{+}(T)\dprod\limits_{j=1}^{m}\left( \int_{K_{j}}\langle x^{j},x_{j}^{\ast
}\rangle ^{p_{j}}d\mu _{j}\right) ^{\frac{1}{p_{j}}}\left( \int_{B_{F^{\ast
\ast }}^{+}}\left\vert \left\langle y^{\ast },y^{\ast \ast }\right\rangle
\right\vert ^{r}d\mu \right) ^{\frac{1}{r}}
\end{eqnarray*}
Consider the operator $u_{j}^{0}:E_{j}\rightarrow L_{p_{j}}\left( K_{j},\mu
_{j}\right) $ defined by 
\begin{equation*}
u_{j}^{0}\left( x^{j}\right) :x_{j}^{\ast }\mapsto x_{j}^{\ast }(x^{j}).
\end{equation*}%
For all $x^{j}\in E_{j}^{+}$ with $1\leq j\leq m$ we have 
\begin{equation*}
\left\Vert u_{j}^{0}\left( x^{j}\right) \right\Vert =\left(
\int_{K_{j}}\langle x^{j},x_{j}^{\ast }\rangle ^{p_{j}}d\mu _{j}\right) ^{%
\frac{1}{p_{j}}}\leq \left\Vert x^{j}\right\Vert .
\end{equation*}%
Let $X_{j}$ be the closure in $L_{p_{j}}(K_{j},\mu _{j})$ of the range of $%
u_{j}^{0}$, and let $u_{j}:E_{j}\rightarrow X_{j}$ be the induced operator.
The operator $u_{j}$ is positive $p_{j}$-summing with $\pi
_{p_{j}}^{+}\left( u_{j}\right) =1$. Let $A_{0}$ be the multilinear operator
defined on $u_{1}^{0}\left( E_{1}\right) \times \ldots \times
u_{m}^{0}\left( E_{m}\right) $ by 
\begin{equation*}
A_{0}\left( u_{1}^{0}\left( x^{1}\right) ,\ldots ,u_{m}^{0}\left(
x^{m}\right) \right) =T\left( x^{1},...,x^{m}\right) .
\end{equation*}%
By $\left( \ref{def2sec5}\right) ,$ we have 
\begin{equation*}
\begin{aligned} & \left| \left\langle A_{0}\left( u_{1}^{0}\left(
x^{1}\right) ,\ldots ,u_{m}^{0}\left(x^{m}\right) \right),y^*\right\rangle
\right| \\&\leq d_{p_1,...,p_m;r}^{+}(T) \prod_{j=1}^{m}\left \Vert
u_{j}^{0}\left( x^{j}\right) \right \Vert \left(
\int_{B_{F^{**}}^{+}}\left\langle y^{*},y^{**}\right\rangle ^r d\mu
\right)^{\frac{1}{r} } . \end{aligned}
\end{equation*}%
Let $A$ be the unique bounded multilinear extension of $A_{0}$ to $%
X_{1}\times \cdots \times X_{m}$. The operator $A$ is Cohen positive
strongly $r^{\ast }$-summing multilinear operator and $d_{r^{\ast
}}^{m+}(A)\leq d_{\left( p_{1},\ldots ,p_{m};r\right) }^{+}(T).$ This
implies that%
\begin{equation*}
d_{r^{\ast }}^{m+}(A)\prod_{j=1}^{m}\pi _{p_{j}}^{+}\left( u_{j}\right) \leq
d_{\left( p_{1},\ldots ,p_{m};r\right) }^{+}(T)
\end{equation*}%
Finally, $T=A\circ \left( u_{1},\ldots ,u_{m}\right) $ with $u_{j}\in \Pi
_{p_{j}}^{+}\left( E_{j};X_{j}\right) ,(1\leq j\leq m)$ and $A\in \mathcal{D}%
_{r^{\ast }}^{m+}(X_{1},...,X_{m};F)$. This completes the proof.
\end{proof}

Any positive $\left( p_{1},\ldots ,p_{m};r\right) $-dominated multilinear
operator can be factorized through a Cohen positive strongly $r^{\ast }$%
-summing multilinear operator and positive $p_{j}$-summing linear operators $%
(1\leq j\leq m)$. Consequently, the class $\mathcal{D}_{\left( p_{1},\ldots
,p_{m};r\right) }^{+}$ forms a positive multi-ideal of type $\mathcal{M}%
_{R}^{+}\left( \mathcal{B}_{1,L}^{+},...,\mathcal{B}_{m,L}^{+}\right) $ where%
\begin{equation*}
\mathcal{M}_{R}^{+}=\mathcal{D}_{r^{\ast }}^{m+}\text{ and }\mathcal{B}%
_{j,L}^{+}=\Pi _{p_{j}}^{+}\left( 1\leq j\leq m\right) .
\end{equation*}

\textbf{Positive }$(p_{1},...,p_{m})$\textbf{-dominated. }A particularly
interesting case of positive $(p_{1},...,p_{m};r)$-dominated operators
occurs when $r=\infty ,$ i.e., $1/p=1/p_{1}+...+1/p_{m}$. These operators
are referred to as \textit{positive }$(p_{1},...,p_{m})$\textit{-dominated}.
We will provide a precise definition of these operators in the context of
mappings from Banach lattices to a Banach space.

\begin{definition}
\label{Def2.13a} Let $1\leq p,p_{1},\ldots ,p_{m}\leq \infty $ with $\frac{1%
}{p}=\frac{1}{p_{1}}+\ldots +\frac{1}{p_{m}}$. Let $E_{1},...,E_{m}$ be
Banach lattices and $Y$ be a Banach space. An $m$-linear operator $%
T:E_{1}\times ...\times E_{m}\rightarrow Y$ is positive $(p_{1},...,p_{m})$%
-dominated, if there is a constant $C>0$ such that for any $%
(x_{i}^{1},...,x_{i}^{m})\in E_{1}^{+}\times ...\times E_{m}^{+}$ ($1\leq
i\leq n$), we have 
\begin{equation}
(\sum_{i=1}^{n}\Vert T(x_{i}^{1},...,x_{i}^{m})\Vert ^{p})^{\frac{1}{p}}\leq
C\prod_{j=1}^{m}\Vert (x_{i}^{j})_{i=1}^{n}\Vert _{p_{j},w}.  \label{ppdm}
\end{equation}%
We denote the space of all such mappings by $\Pi _{p_{1},\ldots
,p_{m}}^{+}\left( E_{1},\ldots ,E_{m};Y\right) $. In this case, we define
the norm 
\begin{equation*}
\pi _{p_{1},\ldots ,p_{m}}^{+}(T)=\inf \{C>0:\quad C\ \text{satisfying\ the\
inequality}\ (\ref{ppdm})\}.
\end{equation*}
\end{definition}

It is straightforward to demonstrate the equivalence of the formula ($\ref%
{ppdm}$) with%
\begin{equation*}
(\sum_{i=1}^{n}\Vert T(x_{i}^{1},...,x_{i}^{m})\Vert ^{p})^{\frac{1}{p}}\leq
C\prod_{j=1}^{m}\Vert (x_{i}^{j})_{i=1}^{n}\Vert _{p_{j},|\omega |},
\end{equation*}%
for any $(x_{i}^{1},...,x_{i}^{m})\in E_{1}\times ...\times E_{m},$ $\left(
1\leq i\leq n\right) $.

\begin{proposition}
The class $(\Pi _{p_{1},\ldots ,p_{m}}^{+},\pi _{p_{1},\ldots ,p_{m}}^{+})$
is a positive Banach right multi-ideal.
\end{proposition}

Similar to the earlier section, we can establish Pietsch's theorem
concerning this class.

\begin{theorem}[Pietsch domination theorem]
\label{thdo2} Let $1\leq p,p_{1},\ldots ,p_{m}<\infty $ with $\frac{1}{p}=%
\frac{1}{p_{1}}+\ldots +\frac{1}{p_{m}}$. Let $E_{1},...,E_{m}$ be Banach
lattices and $Y$ be a Banach space. The following properties are equivalent:

1) The operator $T\in \mathcal{L}\left( E_{1},\ldots ,E_{m};Y\right) $ is
positive $(p_{1},\ldots ,p_{m})$-dominated.

2) There is a constant $C>0$ and Borel probability measures $\mu _{j}$ on $%
B_{E_{j}^{\ast }}^{+}$ ($1\leq j\leq m$), such that 
\begin{equation}
\Vert T(x^{1},...,x^{m})\Vert \leq C\prod_{j=1}^{m}\left(
\int_{B_{E_{j}^{\ast }}^{+}}\langle |x^{j}|,x_{j}^{\ast }\rangle
^{p_{j}}d\mu _{j}\right) ^{\frac{1}{p_{j}}},
\end{equation}%
for all $(x^{1},...,x^{m})\in E_{1}\times ...\times E_{m}.$

3) There is a constant $C>0$ and Borel probability measures $\mu _{j}$ on $%
B_{E_{j}^{\ast }}^{+}$ ($1\leq j\leq m$), such that%
\begin{equation*}
\left\Vert T(x^{1},...,x^{m})\right\Vert \leq C\prod_{j=1}^{m}\left(
\int_{B_{E_{j}^{\ast }}^{+}}\langle x^{j},x_{j}^{\ast }\rangle ^{p_{j}}d\mu
_{j}\right) ^{\frac{1}{p_{j}}},
\end{equation*}%
for all $(x^{1},...,x^{m})\in E_{1}^{+}\times ...\times E_{m}^{+}.$

4) (Kwapie\'{n}'s factorization) There exist Banach spaces $X_{1},\ldots
,X_{m}$, a multilinear operator $A:X_{1}\times ...\times X_{m}\rightarrow Y$
and linear operators $u_{j}\in \Pi _{p_{j}}^{+}\left( E_{j};X_{j}\right) $,
so that $T=A\circ \left( u_{1},\ldots ,u_{m}\right) $, i.e. $\Pi
_{p_{1},\ldots ,p_{m}}^{+}=\mathcal{L}\left( \Pi _{p_{1}}^{+},\ldots ,\Pi
_{p_{m}}^{+}\right) .$ Moreover%
\begin{equation*}
\pi _{p_{1},\ldots ,p_{m}}^{+}(T)=\inf \left\{ \left\Vert A\right\Vert
\prod_{j=1}^{m}\pi _{p_{j}}^{+}\left( u_{j}\right) :T=A\circ \left(
u_{1},\ldots ,u_{m}\right) \right\} .
\end{equation*}%
In other words, we say that the class $\Pi _{p_{1},\ldots ,p_{m}}^{+}=%
\mathcal{L}(\Pi _{p_{1}}^{+},...,\Pi _{p_{m}}^{+})$ is the Banach positive
left multi-ideal generated by the factorization method from the Banach
positive operator left ideals $\Pi _{p_{1}}^{+},...,\Pi _{p_{m}}^{+}.$
\end{theorem}

\begin{remark}
The composition class $\mathcal{D}_{r^{\ast }}^{+}\circ \Pi _{p_{1},\ldots
,p_{m}}^{+}$ is equal to $\mathcal{D}_{r^{\ast }}^{+}\circ \mathcal{L}(\Pi
_{p_{1}}^{+},...,\Pi _{p_{m}}^{+}),$ which in turn is equal to $\mathcal{D}%
_{\left( p_{1},\ldots ,p_{m};r\right) }^{+}$. Consequently, an alternative
expression for the class $\mathcal{D}_{\left( p_{1},\ldots ,p_{m};r\right)
}^{+}$ is given by $\mathcal{B}_{R}^{+}\circ \mathcal{M}_{L}^{+}$ with $%
\mathcal{B}_{R}^{+}=\mathcal{D}_{r^{\ast }}^{+}$ and $\mathcal{M}%
_{L}^{+}=\Pi _{p_{1},\ldots ,p_{m}}^{+}$.
\end{remark}

\textbf{Declarations}\newline

\textbf{Conflict of interest.} The authors declare that they have no
conflicts of interest.\newline


\begin{thebibliography}{99}
\bibitem{Ach11} \textsc{D. Achour, }\textit{Multilinear extensions of
absolutely }$(p;q;r)$\textit{-summing operators}{.} Rend. Circ. Mat. Palermo
60, 337-350 (2011).

\bibitem{AA10} \textsc{D. Achour, A. Alouani, }\textit{On multilinear
generalizations of the concept of nuclear operators}. Colloq. Math 120(1),
85-102 (2010).

\bibitem{AB14} \textsc{D. Achour and A. Belacel, }\textit{Domination and
factorization theorems for positive strongly }$p$\textit{-summing operators}%
. Positivity 18, 785-804 (2014).

\bibitem{AchMez07} \textsc{D. Achour and L. Mezrag, }\textit{On the Cohen
strongly }$p$\textit{-summing multilinear operators}{.} J. Math. Anal. Appl
327, 550-563 (2007).

\bibitem{Bla87} \textsc{O. Blasco, }\textit{Positive }$p$\textit{-summing
operators on }$L_{p}$\textit{-spaces}. Proceedings of the American
Mathematical Society 100.2, 275-280 (1987).

\bibitem{BB18} \textsc{A. Bougoutaia and A. Belacel, }\textit{Cohen positive
strongly }$p$\textit{-summing and }$p$\textit{-convex multilinear operators}%
. Positivity 23.2, 379-395 (2019).

\bibitem{BBH21} \textsc{A. Bougoutaia, A. Belacel and H. Hamdi, }\textit{%
Domination and Kwapi\'{e}n factorization theorems for positive Cohen nuclear
linear operators}. Moroccan Journal of Pure and Applied Analysis 7.1,
100-115 (2021).

\bibitem{BBR23} \textsc{A. Bougoutaia, A. Belacel and P. Rueda, }\textit{%
Summability of multilinear operators and their linearizations on Banach
lattices}. J. Math. Anal. Appl. 527, 127459, (2023).

\bibitem{BPR07} \textsc{G. Botelho, D. Pellegrino and P. Rueda, }\textit{On
composition ideals of multilinear mappings and homogeneous polynomials}.
Publications of the Research Institute for Mathematical Sciences 43.4,
1139-1155, (2007).

\bibitem{CBD21} \textsc{D. Chen, A. Belacel, J. A. Ch\'{a}vez-Dom\'{\i}%
nguez, }\textit{Positive }$p$\textit{-summing operators and disjoint }$p$%
\textit{-summing operators}. Positivity 25, 1045-1077, (2021).

\bibitem{Coh73} \textsc{J.S. Cohen, }\textit{Absolutely }$p$\textit{%
-summing, }$p$\textit{-nuclear operators and their conjugates}{.} Math. Ann
201, 177-200 (1973).

\bibitem{defl} \textsc{A. Defant and K. Floret, }Tensor Norms and Operator
Ideals. North-Holland Publishing, North-Holland, 1993.

\bibitem{DJT95} \textsc{J. Diestel, H. Jarchow and A. Tonge, }Absolutely
Summing Operators. Cambridge University Press, Cambridge, 1995.

\bibitem{Gei84} \textsc{S. Geiss, }Ideale multilinearer Abbildungen,
Diplomarbeit, 1984.

\bibitem{MN91} \textsc{P. Meyer-Nieberg, }Banach lattices. Springer, Berlin,
1991.

\bibitem{PelDoctorat} \textsc{D. Pellegrino, }Ideais de Aplica\c{c}\~{o}es
Multilineares e Polin\^{o}mios entre Espa\c{c}os de Banach. Disserta\c{c}%
\~{a} apresentada ao Departamento de Matem\'{a}tica da Universidade Federal
da Para\'{\i}ba, como requisito parcial para a obten\c{c}\~{a}o do t\'{\i}%
tulo de Mestre em Matem\'{a}tica. 2008.

\bibitem{PSS12} \textsc{D. Pellegrino, J. Santos\ and J.B.S. Sep\'{u}lveda, }%
\textit{Some techniques on nonlinear analysis and applications}{.} Adv. Math
229, 1235-1265 (2012)

\bibitem{Pml} \textsc{A. Pietsch, }Ideals of multilinear functionals. In:
Proceedings of the second international conference on operator algebras,
ideals, and their applications in theoretical physics (Leipzig, 1983). vol.
67, 185-199 (1983).

\bibitem{PIETSCHoi} \textsc{A. Pietsch, }Operator Ideals. North-Holland
Publications, North-Holland, 1980.

\bibitem{Zhu98} \textsc{O.I. Zhukova, }\textit{On modifications of the
classes of }$p$\textit{-nuclear, }$p$\textit{-summing, and }$p$\textit{%
-integral operators}. Sib Math J 39, 894-907 (1998).
\end{thebibliography}
\end{document}